Template for the submission to:
\documentclass[aos]{imsart}

\RequirePackage{amsthm,amsmath,amsfonts,amssymb}
\RequirePackage[authoryear]{natbib} 
\RequirePackage[colorlinks,citecolor=blue,urlcolor=blue]{hyperref}

\usepackage{multirow}
\usepackage{tabularx}

\startlocaldefs
\theoremstyle{plain}
\newtheorem{thm}{Theorem}[section]
\newtheorem{lemma}{Lemma}[section]
\newtheorem{cor}[thm]{Corollary}
\newtheorem{prop}[thm]{Proposition}
\theoremstyle{remark}
\newtheorem*{assumption*}{Assumption}
\newtheorem{condition}{Condition}[section]

\newcommand{\mbf}{\mathbf}

\newcommand{\sbf}{\boldsymbol}

\newcommand{\bbG}{\mathbb{G}}

\newcommand{\bbR}{\mathbb{R}}

\endlocaldefs

\begin{document}

\begin{frontmatter}
\title{Asymptotic Properties of Penalized Spline Estimators in Concave
  Extended Linear Models: Rates of Convergence}
\runtitle{Rates of Convergence of Penalized Splines}

\begin{aug}
\author[A]{\fnms{Jianhua Z.} \snm{Huang} \thanksref{T1}
   \ead[label=e1]{jianhua@stat.tamu.edu}},
 \author[B]{\fnms{Ya} \snm{Su} \ead[label=e2]{suyaf@vcu.edu}}
\thankstext{T1}{Corresponding Author}
\address[A]{Department of Statistics, Texas A\&M University, College Station, TX 77843-3143,
\printead{e1}}
\address[B]{Department of Statistical Sciences and Operations
  Research, Virginia Commonwealth University, Richmond, VA 23284-3083, 
\printead{e2}}
\end{aug}

\begin{abstract}
This paper develops a general theory on rates of convergence of
penalized spline estimators for function estimation when the
likelihood functional is concave in candidate functions, where the likelihood is
interpreted in a broad sense that includes conditional likelihood,
quasi-likelihood, and pseudo-likelihood. The theory allows all
feasible combinations of the spline degree, the penalty order, and 
the smoothness of the unknown functions. According to this theory,
the asymptotic behaviors of the penalized spline estimators
depends on interplay between the spline knot number and 
the penalty parameter. The general theory is applied to obtain results
in a variety of contexts, including regression,
generalized regression such as logistic regression and Poisson regression, density
estimation, conditional hazard function estimation for censored data, quantile regression, diffusion
function estimation for a diffusion type process, and
estimation of spectral density function of a stationary time
series. For multi-dimensional function estimation, the theory
(presented in the Supplementary Material) covers
both penalized tensor product splines and penalized bivariate splines
on triangulations.
\end{abstract}



\end{frontmatter}


\section{Introduction}
Since the publication of the \emph{Statistical Science} discussion
paper of \citet{eilers1996flexible}, penalized spline estimators (or
penalized splines for short) have
gained much popularity and have become a standard general-purpose
method for function estimation. Many applications of penalized splines are presented in the monograph
\citet{ruppert2003semiparametric}.
As an indication of popularity of penalized splines, a google search on ``penalized splines''
yields more than 200,000 results, and the  \citet{eilers1996flexible}
paper has more than 3000 citations.
Despite the popularity of penalized splines, theoretical
understanding of the method falls much behind.
Existing results on asymptotic behaviors of penalized splines
have focused on the nonparametric regression setting. Since 
application of penalized splines has gone far beyond nonparametric regression,
there is a big gap between theory and practice that needs to be filled
in.

\citet{hall2005theory} obtained the asymptotic mean squared error
of penalized spline estimators under a white noise model.
\cite{li2008asymptotics}, \citet{wang2011asymptotics} and \cite{schwarz2016} showed that penalized spline estimators are approximately equivalent to kernel
regression estimators and used this connection to obtain asymptotic
properties of penalized spline estimators.
\cite{claeskens2009asymptotic} and \cite{xiao2019} obtained asymptotic
results for penalized splines under weaker conditions than previously
used in the literature, and identified a breakpoint in rates of convergence to classify two asymptotic situations for penalized
splines, one close to smoothing splines, and one close to polynomial
splines.
Results on estimation of bivariate functions have been obtained by \citet{lai2013bivariate}
for penalized bivariate splines on triangulations, and by
\citet{xiao2013} and \citet{xiao2019jnps} for penalized tensor
product splines with different choices of penalty functionals.
\cite{holland2017penalized} studied asymptotic behaviors of penalized
tensor product splines for estimating multi-dimensional functions.
While the above papers focused on least squares regression, \cite{kauermann2009some} obtained asymptotic behaviors of
penalized spline estimators for generalized regression when the
regressor is univariate.

Most of the works mentioned above have used closed-form expressions of
penalized spline estimators which are only available in the regression
setting. When such expressions are not available in other estimation contexts,
such as estimation of density functions or conditional quantile
functions, existing asymptotic approaches cannot be easily extended,
imposing a challenge on studying the asymptotic behaviors of penalized
splines beyond nonparametric regression.

The goal of this article is to develop a new asymptotic approach to
penalized spline estimators that allows us to obtain general rates of convergence results in a broad range of
contexts, called concave extended linear models
\citep{huang2001concave}. We use the term ``concave extended linear models'' because in all these
contexts, the unknown function is searched over a linear function
space using a maximum-likelihood-type criterion, while the
``likelihood'' is a concave functional of candidate functions.
As we shall see later, the family of concave extended linear models is rich, covers many
useful contexts of function estimation as special cases, including regression,
generalized regression such as logistic regression and Poisson regression, density
estimation, conditional hazard function estimation for censored data, diffusion
function estimation for a diffusion process, quantile regression, and
estimation of spectral density function of a stationary time series.
For readability of the paper, we present only results for univariate
function estimation in the main paper.
Results for multi-dimensional function estimation are obtained in the
same theoretical framework, but will be presented in the 
Supplementary Material since they involve more
complicated notations and background on multivariate splines.

\subsection{Concave extended linear models} \label{ELM_PSE}

Suppose we are interested in estimating an unknown function $\eta_0$
that is associated with the distribution of a random variable or vector
$\mathbf{W}$. This function is defined on a compact set $\mathcal{U}$, which for
concreteness is assumed to be an interval $[a,b]$. We have available an i.i.d.\
sample of $\mathbf{W}$ of size $n$, denoted as $\mathbf{W}_1, \ldots, \mathbf{W}_n$.
For a candidate function $h$ of estimating $\eta_0$, the (scaled) log-likelihood is
\begin{equation}\label{eq:log-lik}
\ell(h;\mathbf{W}_1, \ldots, \mathbf{W}_n) =\frac{1}{n} \sum_{i=1}^nl(h;\mathbf{W}_i),
\end{equation}
where $l(h;\mathbf{W}_i)$ is the contribution to the log-likelihood
from $\mathbf{W}_i$, and the scaling is given by the factor $1/n$. 
The expected log-likelihood is 
\[
\Lambda(h) = E \{\ell(h;\mathbf{W}_1, \ldots, \mathbf{W}_n)\}\; (= E\{
l(h; \mathbf{W}_i)\} \text{ if $\mathbf{W}_i$ are i.i.d.}),
\]
where the expectation is taken with respect to the distribution
of $\mathbf{W}_1, \ldots, \mathbf{W}_n$. For the rest of the paper, when
there is no confusion, we will omit $\mathbf{W}_1, \ldots,
\mathbf{W}_n$ in the log-likelihood expresion and write $\ell(\eta)$ to simplify notation.

Assume that the set of functions for which both the log-likelihood 
and the expected log-likelihood are well-defined is a convex set.
We say that we have a \emph{concave extended linear model} if
\begin{itemize}
\item[(i)] $\ell(h; \mathbf{W}_1, \ldots, \mathbf{W}_n)$ is a concave
in $h$ for all possible values of $\mathbf{W}_1, \ldots,
\mathbf{W}_n$, i.e., for $0\leq
 \alpha \leq 1$, 
\begin{align*}
& \ell(\alpha h_1 +(1-\alpha)h_2; \mathbf{W}_1, \ldots, \mathbf{W}_n) \\
& \qquad \geq \alpha \ell(h_1; \mathbf{W}_1, \ldots, \mathbf{W}_n)
+ (1-\alpha) \ell(h_2; \mathbf{W}_1, \ldots, \mathbf{W}_n);
\end{align*}
\item[(ii)] $\Lambda(h)$ is strictly concave in $h$, i.e., for $0\leq
 \alpha \leq 1$, 
\[
\Lambda(\alpha h_1 +(1-\alpha)h_2) \geq \alpha \Lambda(h_1)
+ (1-\alpha) \Lambda(h_2),
\]
and if $0<\alpha <1$, the strict inequality holds only when it does
not hold that $h_1=h_2, a.e.$.
\end{itemize}

In our framework, the functional $\ell(h)$ can be something more general than
the log-likelihood function. All we need is that the function of
interest, $\eta_0$, maximizes $\Lambda(h)$.
For example, for the regression problem, our goal is to estimate the conditional 
mean $\eta_0(x) = E (Y| X=x)$, by setting $\mathbf{W}_i = (X_i,Y_i)$ 
and 
\[
\ell(h; \mathbf{W}_1, \ldots, \mathbf{W}_n) = - \frac{1}{n} \sum_{i=1}^n \{Y_i - h(X_i)\}^2,
\]
we obtain a concave extended linear model. If the conditional
distribution of $Y_i$ given $X_i$ is Gaussian, $\ell(h; \mathbf{W}_1, \ldots, \mathbf{W}_n)$ can be
interpreted (up to a scale factor) as the log-conditional likelihood, but this distribution
assumption is not needed when applying our results in this paper.
For the problem of estimating a probability density function $\eta_0$,
by setting $\mathbf{W}_i = X_i$ and 
\[
\ell(h; \mathbf{W}_1, \ldots, \mathbf{W}_n) = \frac{1}{n} \sum_{i=1}^n h(X_i) -  \log \int_{\mathcal{U}} \exp h(x)\, dx,
\] 
we also obtain a concave extended linear model. More detailed
discussions of log-likelihood function for a variety of
contexts can be found in
Sections~\ref{sec:regression}--\ref{sec:quantile} and Sections S.2--S.3
in the Supplementary Material.

\subsection{Penalized spline estimators}
For sample size $n$, consider a finite-dimensional space $\bbG_n$ of spline functions with degree $m$. The penalized spline
estimator $\hat{\eta}_n$ is defined as the maximizer among $g\in \bbG_n$ of the following penalized likelihood
\begin{equation}~\label{eq:plik}
p\ell(g; \mathbf{W}_1, \ldots, \mathbf{W}_n) 
= \ell(g; \mathbf{W}_1, \ldots, \mathbf{W}_n) - \lambda_n J_q(g),
\end{equation}
where $\ell(g; \mathbf{W}_1, \ldots, \mathbf{W}_n)$ is the
log-likelihood defined in \eqref{eq:log-lik},
$J_q(g)$ is a penalty term, and $\lambda_n$ is a penalty parameter.
The penalty term $J_q(g) = J_q(g, g)$ is chosen to be a quadratic functional
that quantifies the roughness of a candidate function $g$, and we
use the following specific form in this paper
\begin{equation}\label{Osullivan}
  J_q(g) = \int_{\mathcal{U}} \{g^{(q)}(x)\}^2 \,dx.
\end{equation}
We let $q$ be an integer and refer to it as the order of the penalty.
This kind of estimator was first introduced in \cite{o1986statistical}, \citet{o1988fast},
and later popularized by \cite{eilers1996flexible} where a modified
penalty functional is used. A multi-dimensional analog of the penalty
functional \eqref{Osullivan} is given in the Supplementary Material.

If we do not restrict the maximization to a finite-dimensional space
in the optimization problem~\eqref{eq:plik}, we perform the optimization
over the set of all functions such that the penalty functional is
finite, which is the usual Sobolev space of order $q$
\[
W^q[a,b]= \{h: h^{(q-1)} \text{ is absolutely continuous and } 
J_q(h) <\infty \},
\] 
where $h^{(l)}$ denotes the $l$-th derivative of $h$, then the
resulting estimator is a smoothing spline \citep{wahba1990spline, gu2013smoothing}. 
If there is no penalty term in \eqref{eq:plik} (or $\lambda_n=0$), we
call the resulting estimator a polynomial spline estimator.
In the literature, a polynomial spline function estimator is usually called
a regression spline estimator mainly because the regression problem is
where such an estimator was first applied to, but we prefer the former name because its
application goes far beyond the regression problem. There is an extensive
literature on the asymptotic theory of the smoothing spline estimators and the
polynomial spline estimators, which is reviewed in the
Supplementary Material.

\subsection{Overview of results in this paper}
For the penalized spline estimator $\hat\eta_n$, we obtain a
probabilistic bound on the quantity $\|\hat\eta_n - \eta_0\|^2 +
\lambda_n J(\hat\eta_n)$, where $\|\cdot \|$ is a norm that is equivalent
to the usual $L_2$-norm. Our result not only gives the $L_2$ rate
of convergence of $\hat\eta_n$ to the true function $\eta_0$, but also gives a bound on
$J(\hat\eta_n)$, which measures the roughness of the estimator.

In the framework of concave extended linear models, we establish asymptotic
results for penalized spline estimators under a set of high level
conditions. These high level conditions help us identify the essential
factors governing the asymptotic behaviors, namely, the property of the likelihood, the approximation
property of the spline space, and the eigenstructure of the penalty
functional. Using high level conditions allows us to obtain
results in a unified manner for a wide range of problems,
including the following:
\begin{itemize}
   \item regression (Section~\ref{sec:regression}),
   \item generalized regression (Section~\ref{sec:gr}),
   \item estimation of probability density function (Section~\ref{prob_density}),
   \item hazard regression for censored data (Section~\ref{sec:countingp}),
   \item quantile regression (Section~\ref{sec:quantile}), 
   \item estimation of drift coefficient of diffusion type process
     (Section~S.2),
   \item spectral density estimation for a stationary time series
       (Section~S.3).  
\end{itemize}
To our knowledge, our treatment of rates of convergence for the penalized
spline estimators is the most comprehensive in its
ability to handle a variety of estimation contexts under weak
assumptions.
Using high level conditions allows us to obtain results without
making the strong assumption of equally-spaced knots as used by some existing works.
Our results for the later five contexts are entirely new to the
literature.

Our theory shows that the asymptotic behaviors of penalized splines are
governed by the spline degree $m$, penalty order $q$, degree of
smoothness of the unknown function $p$ (usually denoting the number of
derivatives), and the interplay between the number of knots and the penalty parameter. 
Our results cover all feasible combinations of $m$, $p$, and $q$,
while all existing works only cover selected combinations and are
obtained only in the regression or generalized regression setting.
Following our main results (Sections~\ref{sec:main}), the rates of
convergence of penalized splines can be
classified into seven scenarios and in six of these scenarios the
optimal rate of convergence can be achieved when the spline knot number and the penalty parameter are appropriately chosen
(Table~\ref{summary_table} and its discussion, Section~\ref{sec:summary}).

Our technical approach uses functional analysis tools and
avoids the detailed calculations that involve explicit
expressions of penalized spline estimators as typically used in previous works.
This functional analysis approach is particularly powerful in dealing with
new challenges encountered
when obtaining asymptotic behaviors of penalized splines beyond the regression setting.
For example, one needs to handle the integration-to-one constraints for density estimation,
the non-negative constraint for hazard function estimation, and
non-differentiability of the ``log-likelihood'' for quantile regression. Since penalized spline estimators do not
have a closed-form expression in general settings, the asymptotic approaches
previously used for the regression setting do not apply.
The functional analysis approach also allows us to treat penalized
univariate splines, penalized (multivariate) tensor product splines, and penalized
bivariate splines on triangulations in a unified framework.

Our technical approach has its roots in previous works for obtaining
asymptotic behaviors of
(un-penalized) polynomial
spline estimators, as originated by Charles J.\ Stone in a series of works, synthesized in   
\citet{stone1994use} and \citet{hansen1994extended}, and matured in
\citet{huang2001concave}. As such, we are able to obtain existing results
for polynomial spline estimators as a special simplification of our
approach. On the other hand, considering a penalized likelihood in extended
linear models with a roughness penalty is a substantial
advancement over existing works. We
obtain a rich collection of new results that reveal interesting asymptotic
behaviors of penalized spline estimators that were not
anticipated by \citet{huang2001concave}. We also extended previous theory to deal
with some contexts that were not covered by the framework of \citet{huang2001concave},
such as quantile regression and spectral density estimation.

The rest of the paper is organized as follows. (Sections labeled with S. are in the Supplementary Material.)
Section~\ref{sec:prelim} collects some known facts on the properties
of univariate spline functions and the penalty functional to make this paper self-contained.
Sections~\ref{sec:main} and \ref{sec:proof} present respectively two master theorems and their proofs.
Section~\ref{sec:strategy} gives several lemmas for assisting verification of the
conditions used in the master
theorems. Sections~\ref{sec:regression}--\ref{sec:quantile} and S.2--S.3
verify those conditions under primitive conditions in a variety of contexts. Section~S.1 provides a
literature review on the asymptotic theory of smoothing spline estimators
and polynomial spline estimators. Sections~S.4 and S.5
respectively present our theory for penalized tensor-product
splines and for penalized bivariate splines on triangulations.

\subsection{Notation}
For two real numbers $a$ and $b$, let $a\wedge b$ and $a \vee b$
denote respectively the smaller and larger one of the two. 
Given two sequences of positive numbers $a_n$ and
$b_n$, we write $a_n\lesssim b_n$ and $b_n \gtrsim a_n$ if the ratio $a_n/b_n$ is bounded for all $n$
and $a_n \asymp b_n$ if and only if $a_n\lesssim b_n$ and $b_n
\lesssim a_n$,
we write $a_n \prec b_n$ and $b_n \succ a_n$ if $a_n/b_n \to 0$ as $n\to\infty$.
Let $\|g\|_{2}$ denote the $L_2$-norm (relative to the Lebesgue
measure) and $\|g\|_\infty$ the $L_\infty$-norm of the function $g$. 
Throughout the paper, we use $C$, $M$, and possibly with subscripts to denote constants
whose values may vary from contexts to contexts.

\section{Preliminaries: splines and penalty functionals}\label{sec:prelim}
This section provides the necessary background about spline functions
and penalty functionals, introduces notations, and presents some general
assumptions. In particular, it summarizes some key results from the literature
about spline functions and the penalty functionals, which are essential
for our study of asymptotic properties of the penalized spline estimators.

\subsection{Splines}\label{sec:splines}

A spline function is a numerical function that is piecewise-defined by
polynomial functions, and the polynomial pieces are connected smoothly.
More precisely, a function $f$ defined on a compact interval
$[a,b]$ is called a spline of degree $m$ with $k$ interior
knots $t_j, j=1, \dots, k$ (satisfying $a=t_0<t_1<\dots, t_k
<t_{k+1}=b$), if $f$ is a polynomial of degree $m\geq 0$ on $[t_j,t_{j+1}],
j=0, \dots, k$, and $f$ globally has $m-1$ continuous derivatives (no
derivative if $m=0$).
Note that, for a given sequence of knots, the collection of all
degree-$m$ splines on $[a,b]$ forms a linear vector space
with dimension $N= m+k+1$, denoted as $\mathbb{G}$.

When we study the asymptotic properties of penalized spline
estimators, we allow the number of knots to increase with
the sample size. We write $N=N_n$ and $\mathbb{G} =\mathbb{G}_n$ to make this dependence explicit.
We assume that the knot sequence has the bounded mesh ratio.
More precisely, we assume that the ratio of the maximum and 
minimum distance between two neighboring knots is bounded from 
above and below by two positive numbers that do not depend on $n$, that is,
\begin{equation*}
  C_1 \leq \frac{\max_j(t_{j+1} - t_j)}{\min_j(t_{j+1} - t_j)} \leq C_2, \quad \text{
    for some } C_1, C_2 > 0.
\end{equation*}
Let $\delta_n$ be the largest distance between all the neighboring knots, that
is,
\begin{equation}
  \delta_n = \max_{j} |t_{j+1} - t_j|.  
\end{equation}
Under the assumption of bounded mesh ratio, we have $\delta_n \asymp {1}/{N_n}$.

The rationale for using splines in function estimation is that
splines have a good approximation property, namely, they can approximate
smooth functions very well when the knot number increases to infinity,
as shown in the next result \citep[Theorem 6.25 and Corollary 6.26 of][]{schumaker1981spline}.
(Please note difference in notation. We state the result in terms of
spline degree, while the result in the cited book is stated using the
order of splines.)
 
\begin{prop}\label{schumaker}
Assume $\eta_0 \in W^p[a,b] $ and $m \geq {p}-1$.  
There exist a function $\eta^*_n \in \mathbb{G}_n$ and
constants $C_1$-$C_3$,  depending on $p$ and $\eta_0$ such that
\begin{equation*}
 \|\eta^*_n - \eta_0\|_2 \leq C_1 \delta_n^{p}, \qquad  
\|\eta^*_n - \eta_0 \|_\infty \leq C_2 \delta_n^{p-1/2},
\end{equation*}
and moreover, if $q\leq m$, then  $J_{q}(\eta^*_n) \leq C_3 \delta_n^{2(p - q) \wedge 0}.$
\end{prop}

If $m< p-1$, since  $\eta_0 \in W^p[a,b] $ implies that $\eta_0 \in
W^{m+1}[a,b]$, the conclusion of this theorem holds by replacing $p$ with
$m+1$. This approximation rate is the best one can expect: the
approximation error rate cannot be better than $\delta_n^{m+1}$ even
when the function $\eta_0$ has smoothness $p > m+1$, as shown in
Theorem~6.42 of \citet{schumaker1981spline}. Because of the saturation
phenomenon of the spline approximation, we define $p' = p \wedge
(m+1)$ and use $p'$ to measure the rate of approximation
error. Moreover, we will require later that $p>1/2$ in order to guarantee $\|\eta^*_n - \eta_0\|_\infty =o(1)$.

Following \cite{huang1998projection, huang1998functional}, we 
introduce a measure of the complexity of a spline space,
\begin{equation}
  \label{A_n}
  A_n = \sup_{g \in \bbG_n, \|g\|_{2} \neq 0}\biggl\{\frac{\|g\|_\infty }{\|g\|_{2}}\biggr\}.
\end{equation}
This measure will play an important role in the asymptotic analysis.
The next result, which follows from Theorem 5.1.2 of
\cite{devore1993constructive}, gives the rate of increase of $A_n$.

\begin{prop}\label{bspline_constant}
Under the bounded mesh ratio condition, $A_n \asymp \delta_n^{-1/2}$.
\end{prop}

The asymptotic analysis of spline estimators relies on an important
property of spline spaces, namely, the uniformly closeness of a
data-driven norm to its expectation over the entire spline space
for a fixed degree and fixed knot sequence (they vary with $n$). 
Let $X, X_1, \dots, X_n$ be i.i.d. random variables.
Define the empirical and theoretical inner products as
\begin{align*}
\langle g_1, g_2\rangle_n & = E_n[g_1(X) g_2(X)\, w(X)] = \frac{1}{n}\sum_{i=1}^n g_1(X_i)
g_2(X_i)w(X_i), \\
\langle g_1, g_2\rangle & = E[g_1(X) g_2(X)\, w(X)],
\end{align*}
where $w(x)$ is a weight function bounded away from zero and infinity, that
is, there exists $C_1, C_2 > 0$ such that
\[
    C_1 \leq w(x) \leq C_2, \quad \text{for any} \quad a \leq x \leq b. 
\]
The corresponding squared empirical and theoretical norms are
$\|g\|^2_n = \langle g, g\rangle_n$ and $\|g\|^2 = \langle g,g\rangle$.
We assume that $X$ has a density function which is bounded away from
0 and infinity, and consequently the theoretical norm $\|\cdot\|$ is equivalent 
to $\|\cdot\|_2$, the usual $L_2$-norm relative to the Lebesgue
measure, i.e., there are constants $C_3$ and $C_4$ such that
$ C_3 \|g\|_2 \leq \|g\| \leq C_4 \|g\|_2$ for all square-integrable
function $g$.

\begin{prop}\label{equiv_norm}
  Under the bounded mesh ratio condition, if \\
$\lim_n N_n \log(n) /n = 0$, then the
  empirical and theoretical norms are asymptotically equivalent, that
  is,
 \[
\sup_{g \in \bbG_n, \|g\|\not = 0} \bigg| \frac{\|g\|_n}{\|g\|} - 1\biggr| = o_P(1).
\]
\end{prop}

\cite{huang1998projection} proved Proposition \ref{equiv_norm} for
an arbitrary finite dimensional function space under the stronger
  condition that $\lim_n A_n^2 N_n /n = 0$. 
\cite{huang2003asymptotics}  relaxed the condition to
$\lim_n N_n \log(n) /n = 0$ for splines. Both papers proved the
results for $w(x)=1$ but the same argument applies to a general weight
function that is bounded away from zero and infinity.

\subsection{The penalty functional}
The asymptotic properties of the penalized spline estimator rely
heavily on the eigenanalysis of the quadratic penalty functional
$J_q(h) = \int_\mathcal{U} \{h^{(q)}(x)\}^2 dx$ with respect to the
quadratic functional $V(h) = \|h\|^2 = \int_\mathcal{U} h^2(x) \omega(x) dx$.
Such eigenanalysis also plays a critical role in studying the asymptotic
properties of the smoothing splines; see, i.e., \cite{gu2013smoothing}.

A quadratic functional $B$ is said to be completely continuous with
respect to another quadratic functional $A$, if for any $\epsilon >0$, 
there exists a finite number of linear functionals
$L_1,\cdots,L_k$ such that $L_1(h)=\cdots=L_k(h)=0$ implies that
$B(h) \leq \epsilon A(h)$. See \citet[][Section 3.3]{weinberger1974variational}. 

Applying Theorem 3.1 of \cite{weinberger1974variational}, it can be
shown that, if $V$ is
completely continuous with respect to $J$, then $V$ and $J$ can be
simultaneously diagonalized in the following sense 
\citep[see Section 9.1 of][]{gu2013smoothing}. There exists a sequence of eigenfunctions $\phi_\nu$, $\nu =1, 2, \dots, $ and the associated sequence of eigenvalues
$\rho_\nu\geq 0$ of $J$ with respect to $V$ such that
\[
V(\phi_\nu,\phi_\mu) = \delta_{\nu\mu}, \quad 
J_q(\phi_\nu,\phi_\mu) = \rho_\nu \delta_{\nu\mu},
\]
where $\delta_{\nu\mu}$ is the Kronecker delta, 
\[
V(\phi_\nu,\phi_\mu) = \int_\mathcal{U} \phi_\nu(x) \phi_\mu(x)
\omega(x) \, dx, \quad
J_q(\phi_\nu,\phi_\mu) = \int_\mathcal{U} \phi_\nu^{(q)}(x)
\phi_\mu^{(q)}(x)\, dx.
\]
See also \citet{silverman1982estimation}.
Furthermore, any function $h$ satisfying $J_q(h) < \infty$
has a Fourier series expansion with the eigen basis
$\{\phi_\nu\}$,
\[
h = \sum_\nu h_\nu\phi_\nu, \quad h_\nu = V(h,\phi_\nu),
\]
and 
\[
V(h) =\sum_\nu h_\nu^2, \quad J_q(h) = \sum_\nu \rho_\nu h_\nu^2.
\]
Therefore,
\[
\|h\|^2 + \lambda_n J_q(h) = (V + \lambda_n J)(h) 
= \sum_\nu (1 + \lambda_n \rho_\nu) h_\nu^2.
\]

The next result \citep[see (3.17) of][]{utreras2013optimal} gives the rate of divergence to infinity of the eigenvalues.

\begin{prop}
 \label{rate_eigenvalue}
Assume $V(h) = \|h\|^2 = \int_{\mathcal{U}} h^2(x) \omega(x)\,dx$ for a weight
function $\omega$ that is bounded away from zero and infinity, that is, there exist constants $C_1, C_2 > 0$ such that
  \begin{eqnarray*}
    C_1 \leq \omega(x) \leq C_2, \quad \text{for any} \quad a \leq x \leq b. 
  \end{eqnarray*} 
Then $V$ is completely continuous with respect to $J_q$.
Moreover, we have $0\leq \rho_\nu \uparrow \infty$, and
$\rho_\nu\asymp \nu^{2q}$ for all sufficiently large $\nu$.
\end{prop}

The following result, which is part of Lemma~9.1 of \cite{gu2013smoothing}, will
be used when studying the rate of convergence of the estimation error
(see Lemma~\ref{lemma:score}).

\begin{prop}\label{gu1}
Assume there is a constant $C>0$ such that $\rho_\nu \ge C \nu^{2q} \,
(q> 1/2), $ for all large $\nu$. If $\lambda_n
\rightarrow 0$, as $n\to \infty$, then
\begin{equation}\label{appx-equality}
\sum_{\nu}\frac{1}{1+\lambda_n\rho_{\nu}}=O(\lambda_n^{-1/(2q)}).
\end{equation}
\end{prop}

\section{Statement of the master theorems}\label{sec:main}
The rate of convergence of a penalized spline estimator depends on
three quantities:
\begin{itemize}
\item $p$---the smoothness $p$ of the unknown function (i.e., we
  assume $\eta_0\in  W^p[a,b]$);
\item $m$---the degree of the splines in $\bbG_n$;
\item $q$---the order of the penalty functional $J_q(g) = \int_{\mathcal U}
  \{g^{(q)}(x)\}^2\, dx$.
\end{itemize}
Here, $m+1$ is also called the order of the spline functions.

We make several (natural) restrictions on the choice of $p$, $m$, $q$,
as follows.
\begin{itemize}
\item $q \leq m$.
Since the $m$th derivative of a spline function of degree $m$ is
piecewise constant, the $(m+1)$th derivative of the spline function contains
Dirac delta functions, therefore the $(m+1)$th order penalty
functional is not defined, thus it is natural to require that $q \leq m$.
\item $p>1/2$.
This is to ensure that the spline space has desired approximation
properties (see Proposition~\ref{schumaker}). 
\item $q>1/2$.
This is to ensure the eigenvalues of the penalty functional have
desired rate of divergence (see Proposition~\ref{gu1}).
\end{itemize}

In this paper, we also restrict $p$ and $q$ to be integer-valued, which is
the most relevant in practical applications. To relax
this restriction, one needs only to supply a version of 
Propositions~\ref{schumaker} and \ref{rate_eigenvalue} that allow non-integer
values of $p$ and $q$. The rest of technical arguments is not affected.

The expected value of the penalized log-likelihood $p\ell(\eta)$ appeared in
\eqref{eq:plik} is
\[
\mathsf{p\Lambda}(\eta) = \Lambda(\eta) - \lambda_n J_q(\eta). 
\]
Denote its maximizers as
\begin{equation}
  \label{etabar_elm}
 \bar{\eta}_n = \text{arg max}_{g \in \bbG_n} \mathsf{p\Lambda}(g) = \text{arg
  max}_{g \in \bbG_n} \{\Lambda(g) - \lambda_n J_q(g)\}. 
\end{equation}
We can think that $\bar\eta_n$ is an approximation of
$\eta_0$, and the penalized spline estimator $\hat\eta_n$ directly estimates $\bar\eta_n$. Therefore we have the decomposition
\begin{equation}\label{decomp_hat_0}
  \hat{\eta}_n - \eta_0 = \hat{\eta}_n - \bar{\eta}_n + (\bar{\eta}_n - \eta_0),
\end{equation}
where $\hat{\eta}_n -\bar{\eta}_n$ and $\bar{\eta}_n - \eta_0$ are
referred to as the estimation error and the approximation error,
respectively.

\subsection{Approximation Error}

\begin{condition}  
\label{cond:exp-loglik}
There are constants $B > 0$ and constants $M_1, M_2>0$ such that 
  \begin{equation}\label{L_diff}
  -M_1 \|h\|^2 \leq \Lambda(\eta_0 + h)
  - \Lambda(\eta_0) \leq -M_2 \|h\|^2   
  \end{equation}
whenever $\|h\|_\infty \leq B$. 
\end{condition}

This condition says that the expected log-likelihood behaves like a
quadratic functional around its maximal point.

Recall $p' = p \wedge (m+1)$, as defined after Proposition~\ref{schumaker}.

\begin{thm}\label{approx_err}
Assume Condition~$\ref{cond:exp-loglik}$ holds.
  If $\lim_n \delta_n \lor \lambda_n =0 $ and 
$$\lim_n A_n^2 \allowbreak \{\delta^{2{p'}}_n \lor (\lambda_n \delta_n^{2(p'-q) \wedge 0})\} = 0,$$
  then $\bar{\eta}_n$ exists and $\|\bar{\eta}_n\|_\infty
  =O(1)$. Moreover, $\|\bar\eta_n - \eta_0\|_\infty = o(1)$ and
\[
\|\bar\eta_n-\eta_0\|^2   + \lambda_n  J_q(\bar{\eta}_n) = O\{\delta^{2{p'}}_n \lor (\lambda_n \delta_n^{2(p'-q) \wedge 0})\}.
\] 
\end{thm}

\subsection{Estimation Error}

To simplify notation, we shall omit $\mathbf{W}_1, \ldots,
\mathbf{W}_n$ when we write the log-likelihood functional
in $\ell(h;\mathbf{W}_1, \ldots, \mathbf{W}_n)$.

Because $l(\bar\eta_n + \alpha g)$ is a concave function of $\alpha$,
it admits left and right derivatives and is differentiable at all but
countable many points. Denote the directional derivative at $\bar\eta_n$ along the direction of
$g$ as
\[
\dot{l}[\bar\eta_n; g] =\frac{d}{d\alpha} l(\bar\eta_n + \alpha g) \bigg |_{\alpha=0^+},
\]
where the dependence $\dot{l}[\bar\eta_n; g]$ on $\mathbf{W}_i$ is suppressed in our notation for simplicity.
Using the mild assumption that we can exchange differentiation and
expectation, we have $E \{\dot{l}[\bar\eta_n; g]\} = (d/d\alpha) \Lambda(\bar{\eta}_n +
    \alpha g)|_{\alpha=0^+}$.

Since $\bar{\eta}_n$ maximizes the concave functional
$\mathsf{p\Lambda}(\cdot)$ over $\bbG_n$, it satisfies the first order condition
  \begin{equation*}
    \frac{d}{d\alpha} \Lambda(\bar{\eta}_n + \alpha g)\bigg|_{\alpha=0^+} - 2
    \lambda_n J_q(\bar{\eta}_n, g) = 0, \qquad g \in \bbG_n.
  \end{equation*}
Thus, for any $g \in \bbG_n$, we have that 
\begin{equation*}
    \frac{d}{d\alpha} \ell(\bar{\eta}_n +
    \alpha g)\bigg|_{\alpha=0^+} - 2\lambda_n J_q(\bar{\eta}_n,g) = \frac{d}{d\alpha}\ell(\bar{\eta}_n +
    \alpha g)\bigg|_{\alpha=0^+} - \frac{d}{d\alpha} \Lambda(\bar{\eta}_n +
    \alpha g)\bigg|_{\alpha=0^+}. 
\end{equation*}
Consequently, we have
\begin{equation}\label{eq:gradient}
\frac{d}{d\alpha} \ell(\bar{\eta}_n + \alpha g)\bigg|_{\alpha=0^+} - 2 \lambda_n
J_q(\bar{\eta}_n,g)
= (E_n - E) \; \dot{l}[\bar\eta_n; g].
\end{equation}

\begin{condition}  \label{l_12der}
$(i)$ 
\begin{equation*}
\sup_{g \in \bbG_n} \;  \frac{|(E_n - E) \; \dot{l}[\bar\eta_n; g]|^2}{\|g\|^2 + \lambda_n
  J_q(g)} = O_P\biggl( \frac{1}{n\lambda_n^{1/(2q)}}
\bigwedge \frac{1}{n\delta_n} \biggr).
\end{equation*}

$(ii)$ There are constants $B>0$ and $M > 0$ such that, with probability tending
to one as $n \to \infty$, we have that for all $g \in \bbG_n$ with
$\|g\|_\infty \leq B$,
  \begin{equation*}
   \frac{d}{d\alpha} \ell(\bar{\eta}_n + \alpha
   g)\bigg|_{\alpha = 1^+} - \frac{d}{d\alpha} \ell(\bar{\eta}_n + \alpha
   g)\bigg|_{\alpha = 0^+} \leq - M \|g\|^2.  
  \end{equation*}
\end{condition}

\begin{thm}
  \label{est_err}
  Assume Condition $\ref{l_12der}$ holds. If $\lim_n \delta_n \lor \lambda_n =0 $ and 
$$\lim_n
  A_n^2 \allowbreak (\frac{1}{n\lambda^{1/(2q)}_n} \wedge \frac{1}{n\delta_n}) = 0,$$
  then $\|\hat\eta_n - \bar\eta_n\|_\infty = o_P(1)$ and 
\[
\|\hat\eta_n - \bar\eta_n\|^2 + \lambda_n J_q(\hat{\eta}_n - \bar{\eta}_n) =
  O_p\biggl(\frac{1}{n  \lambda_n^{1/(2q)}}\bigwedge \frac{1}{n \delta_n} \biggr).
\] 
\end{thm}

\subsection{Summary}\label{sec:summary}
Combining the results of Theorems \ref{approx_err} and \ref{est_err},
we obtain the following result that gives the rate of convergence of
$\|\hat{\eta}_n - \eta_0\|^2$ to zero.
The result also gives a bound for the size of  $J_{q}(\hat{\eta}_n)$.

\begin{cor}\label{cor:conv_rates}
Assume Conditions~$\ref{cond:exp-loglik}$ and $\ref{l_12der}$ hold.
 If $\lim_n \delta_n \lor \lambda_n =0 $ and 
\begin{equation}\label{eq:cond_cor}
\lim_n A_n^2 \biggl (\delta^{2{p'}}_n \lor (\lambda_n \delta_n^{2(p'-q) \wedge 0} )
+ \frac{1}{n\lambda_n^{1/(2q)}} \bigwedge \frac{1}{n\delta_n}\biggr) = 0,
\end{equation}
 then $\|\hat{\eta}_n - \eta_0\|_\infty = o_P(1)$ and
\begin{equation*}
\|\hat{\eta}_n - \eta_0\|^2 + \lambda_n J_{q}(\hat{\eta}_n) 
= O_p\biggl(\delta_n^{2{p'}} \vee  (\lambda_n \delta_n^{2(p'-q) \wedge 0})   + \frac{1}{n
    \lambda_n^{1/(2q)}}\bigwedge \frac{1}{n \delta_n} \biggr).
\end{equation*}
\end{cor}

\begin{table}[t]
\caption{Seven scenarios for rate of convergence {\rm ($\|\hat\eta_n - \eta_0\|^2 + \lambda_n J_q(\hat\eta_n)$)} of penalized spline estimators}
\label{summary_table}
\begin{center}
\begin{tabular}{c >{\centering\arraybackslash}p{5cm} c}
\hline
\textbf{Rate of convergence} &\textbf{Parameters for achieving the
                                                 best rate} & \textbf{Best rate} \\[1em]
\hline\hline
\multicolumn{3} {c}{I. $ q<p'$ (i.e., $q < p$ and $q < m+1$)}\\[1em]
\hline
\multicolumn{3} {l}{1. $\lambda_n \lesssim \delta_n^{2p'}$} \\[.5em]
 $\delta_n^{2p'} + (n\delta_n)^{-1}$
 & $\delta_n\asymp n^{-1/(2p'+1)} $&  $n^{-2p'/(2p'+1)}$  (*) \\
\hline
\multicolumn{3} {l}{2. $\delta_n^{2p'} \lesssim \lambda_n \lesssim
  \delta_n^{2q}$} 
\\[.5em]
 $\lambda_n + (n \delta_n)^{-1}$
 & $\lambda_n \asymp \delta_n^{2p'}$,
$\delta_n\asymp n^{-1/(2p'+1)}$ &  $n^{-2p'/(2p'+1)}$  (*)  \\
\hline
\multicolumn{3} {l}{3. $\lambda_n \gtrsim \delta_n^{2q}$}\\[.5em]
$\lambda_n + (n \lambda_n^{1/(2q)})^{-1}$
& $\lambda_n \asymp n^{-2q/(2q+1)}$ &  $n^{-2q/(2q+1)}$ \\
\hline\hline
\multicolumn{3} {c}{II. $ q= p'(=p)$ (i.e., $p = q \leq m$)}\\[1em]
\hline
\multicolumn{3} {l}{1. $\lambda_n \lesssim \delta_n^{2p}$} \\[.5em]
 $\delta_n^{2p} + (n \delta_n)^{-1}$ 
& $\delta_n \asymp n^{-1/(2p+1)}$& $n^{-2p/(2p+1)}$ (**) \\
\hline
\multicolumn{3} {l}{2. $\lambda_n \gtrsim \delta_n^{2p}$}\\[.5em]
 $\lambda_n  + (n \lambda_n^{1/(2p)})^{-1}$ 
& $\lambda_n \asymp n^{-2p/(2p+1)}$ & $n^{-2p/(2p+1)}$ (**) \\
\hline\hline
\multicolumn{3} {c}{III. $ q > p' (= p)$ (i.e., $p < q \leq m$)}\\[1em]
\hline 
\multicolumn{3} {l}{1. $\lambda_n \lesssim \delta_n^{2q}$}\\[.5em]
 $\delta_n^{2p} + (n \delta_n)^{-1}$ 
& $\delta_n \asymp n^{-1/(2p+1)}$ & $n^{-2p/(2p+1)}$  (**)\\
\hline
\multicolumn{3} {l}{2. $\lambda_n \gtrsim\delta_n^{2q}   $}\\[.5em]
 $ \lambda_n \delta_n^{2p-2q} +  (n \lambda_n^{1/(2q)})^{-1}$ 
&  $\delta_n \asymp \lambda_n^{1/(2q)}$, $\lambda_n \asymp n^{-2q/(2p
  +1)}$& $n^{-2p/(2p+1)}$  (**)\\
\hline
\end{tabular}
\end{center}
(*) achieving Stone's optimal rate when $p'=p$,  (**) achieving Stone's optimal rate 
\end{table}

This result covers all practical combinations of $p$, $q$ and $m$ with
the only restriction being the necessary requirement $q \leq m$ (otherwise the
penalty functional is not defined). Following this result, the
asymptotic behavior of the penalized splines can be classified into
seven scenarios as shown in Table~\ref{summary_table}. Cases II.1 and II.2 contain
a typical application scenario of using penalized cubic splines with a second-order
penalty ($m=3$, $q=2$) to estimate a function with a continuous second
derivative $p=2$. Using Proposition~\ref{bspline_constant}, Condition~\eqref{eq:cond_cor} can be simplified in each
scenario as follows:
\begin{itemize}
\item Cases I.1, II.1, III.1: $p' > 1/2$, $n \delta_n^2 \to \infty$.
\item Case I.2: $n \delta_n^2 \to \infty$, $\lambda_n / \delta_n \to
  0$.
\item Cases I.3, II.2, III.2: $n \delta_n \lambda_n^{1/(2q)}  \to \infty$  (or its
  sufficient condition $n \delta_n^2 \to \infty$), $\lambda_n / \delta_n \to
  0$.
\end{itemize}
An overall sufficient condition for all these conditions to hold is 
$p' > 1/2$, $n \delta_n^2 \to \infty$, and $\lambda_n / \delta_n^{1+
  2(q - p)\wedge 0} \to 0$.

From Table~\ref{summary_table}, we observe that the asymptotic behavior of the penalized
spline estimators depend on the interplays among the smoothness of
unknown function, spline degree, penalty order, spline knot number,
and penalty parameter.

In Cases I.1, I.2, II.1. III.1, $\lambda_n \lesssim \delta_n^{2q}$. 
Since using a small $\lambda_n$ indicates light
penalization, we may refer to these cases as the light
penalty scenarios. Alternatively, since $\delta^{-1}_n \lesssim
\lambda_n^{1/{(2q)}}$ and $\delta_n^{-1}$ essentially quantifies the 
number of knots, we may also refer to these cases as the small knot
number scenarios. The behavior of the penalized splines in these scenarios is similar to that of an
unpenalized polynomial spline estimator
\citep[e.g.,][]{huang2003asymptotics}. In Cases I.1 (if $p\le m+1$),
II.1, III.1, the penalized spline estimator achieves Stone's optimal rate of
convergence $n^{-2p/(2p+1)}$ \citep{stone1982}, if the tuning
parameter $\delta_n$ is chosen such that $\delta_n \asymp
n^{-1/(2p+1)}$. 
In Case I.2 (if $p \le m+1$), Stone's optimal rate can be achieved if
we tune both parameters so that $\delta_n \asymp
n^{-1/(2p+1)}$ and $\lambda_n \asymp \delta_n^{2p}$.
If $p>m+1$  (Cases I.1 and I.2), the best rate of convergence of
penalized spline estimator is controlled by the spline order $m+1$, Stone's optimal rate of convergence cannot
be achieved, as for the unpenalized polynomial spline estimators; this
is due to the saturation of spline approximation (see the discussion
following Proposition~\ref{schumaker}).

In Cases I.3, II.2, III.2,
$\lambda_n \gtrsim \delta_n^{2q}$. We may refer to these cases as the heavy penalty scenarios.
Alternatively, since $\delta^{-1}_n \gtrsim \lambda_n^{1/{(2q)}}$, we may
also refer to these cases as the large knot number scenarios. 
The behavior shown in Case II.2 is similar to that of a smoothing
spline estimator \citep[e.g.,][]{gu2013smoothing} and Stone's optimal rate of
convergence $n^{-2p/(2p+1)}$ can be achieved by choosing $\lambda_n
\asymp n^{-2p/(2p+1)}$. The results for Cases I.3 and III.2 show different behaviors of the
penalized spline estimators in the heavy penalty scenarios when the penalty order $q$ differs from the
smoothness $p$ of the unknown function. If $q<p$ (Case I.3), the best rate of
convergence of penalized spline estimator is controlled by $q$, which
is $n^{-2q/(2q+1)}$ and is slower than Stone's optimal rate
$n^{-2p/(2p+1)}$.
This result suggests that, in heavy penalty scenarios, using a
penalty with order smaller than the true smoothness will hurt the
ability of penalized splines to achieve the optimal rate of
convergence. On the other hand, if $q>p$ (Case III.2), the
penalized spline estimator can achieve Stone's optimal rate
if we tune both parameters so that $\lambda_n \asymp n^{-2p/(2p+1)}$ and $\delta_n
\asymp \lambda_n^{1/(2q)}$.

In the context of least squares regression, rates of convergence for
penalized spline estimators have been extensively studied when $q\leq
p$ (corresponding to Cases I and II in Table 1); the best available results
are given in \citet{claeskens2009asymptotic, holland2017penalized, xiao2019}.
Our results match the best available results for Cases I.1, I.3, II.2, II.3. For
Case I.2, the best available result for rate of convergence is $\lambda_n^2
\delta_n^{-2q} + (n\delta_n)^{-1}$ (e.g., Theorem 1(a) of
\citet{claeskens2009asymptotic} for $p=m+1$,
Theorem 5.1 of \citet{xiao2019}), which is always no larger than the
rate shown in Table I, $\lambda_n + (n\delta_n)^{-1}$. 
When $p\leq m+1$ so that $p'=p$, to achieve
Stone's optimal rate, one needs to choose $\delta_n \asymp n^{-1/(2p+1)}$ in $\lambda_n^2
\delta_n^{-2q} + (n\delta_n)^{-1}$, and also require
that $\lambda_n^2 \delta_n^{-2q} \lesssim n^{-2p/(2p+1)}$, or
equivalently $\lambda_n \lesssim n^{-(p+q)/(2p+1)}$. This requirement
on $\lambda_n$ is slightly looser than our requirement $\lambda_n \lesssim n^{-2p/(2p+1)}$
in Cases I.1 and I.2 of Table 1.

It is worthwhile to point out that our result in Corollary~\ref{cor:conv_rates}
not only bound the squared $L_2$-norm $\|\hat\eta_n - \eta_0\|^2$ but also bound the penalty
functional $J_q(\hat\eta_n)$, and thus it is stronger than existing
results which bound only the $L_2$-norm. For this reason,
we believe our rate of convergence in Case I.2, $\lambda_n
+ (n\delta_n)^{-1}$, cannot be improved to match the best available
result of squared $L_2$-norm rate $\lambda_n^2
\delta_n^{-2q} + (n\delta_n)^{-1}$ mentioned above. To see this, suppose otherwise,  i.e.,
$$\|\hat\eta_n - \eta_0\|^2 + \lambda_n J_q(\hat\eta_n) = O(\lambda_n^2 \delta_n^{-2q} + (n\delta_n)^{-1}).$$
When $\lambda_n^2 \delta_n^{-2q} \geq  (n\delta_n)^{-1}$, the first
term dominates the rate of convergence, and we have $J_q(\hat\eta_n) =
O(\lambda_n \delta_n^{-2q}).$ If $\lambda_n/ \delta_n^{2q} \to 0$ (which
falls in Case I.2), then we obtain $J_q(\hat\eta_n) \to 0$, which is
generally implausible. For instance, $J_q(\hat\eta_n)=0$ for $q=2$
means that $\hat\eta_n$ is a straight line, and $J_2(\hat\eta_n) \to
0$ suggests that $\hat\eta_n$ becomes closer and closer to a straight line
when the sample size $n \to \infty$.

We are not aware any existing results for Cases III.1 and III.2. Our
results for these two scenarios answer the following question: When
the smoothness of the unknown function is not given, if one uses a
penalty that assumes more derivatives than the unknown function, how
will the penalized spline estimator behave asymptotically? 
Our answer is that it does not hurt the ability of
penalized spline estimator to achieve Stone's optimal rate of
convergence. This question is of interest because in practice prior knowledge
about the degree of smoothness of the unknown function is usually
unavailable.

\section{Proof of the master theorems}\label{sec:proof}

This section gives the proof of the main theorems of convergence rates of
the penalized spline estimator, that is, Theorems \ref{approx_err} and
\ref{est_err}. The argument makes use of the convexity and is an
extension of that in \cite{huang2001concave}.
We first present a lemma that will play an important role in our proof.

\begin{lemma}[Convexity Lemma]\label{lemma:convex} Suppose $C(\cdot)$ is a convex functional and $L(\cdot)$ is a continuous functional defined on a convex set $\mathcal{C}$ of functions.\\
If there exists a function $\eta^\dag\in \mathcal{C}$ and a real
number $s$ with $L(\eta^\dag) < s$ such that for all $\eta \in
  \mathcal{C}$ satisfying $L(\eta) = s$, we have either
  \begin{equation}\label{temp_1}
    C(\eta^\dag) < C(\eta),
  \end{equation}
or
  \begin{equation}\label{temp_2}
   \frac{\partial}{\partial \beta}\, C(\eta^\dag + \beta (\eta -
   \eta^\dag) ) \bigg|_{\beta=1^+}>0,
  \end{equation}
then any minimizer $\eta_{\mathrm{min}}$ of
$C(\cdot)$ in $\mathcal{C}$ satisfies $L(\eta_{\mathrm{min}}) \leq s$. 
\end{lemma}

\begin{proof}
Fix any $\tilde{\eta} \in \mathcal{C}$ with $L(\tilde{\eta}) > s$.
Consider the convex combination of $\eta^\dag$ and $\tilde{\eta}$ 
\[
\eta_\alpha = \alpha \tilde\eta + (1 - \alpha) \eta^\dag, \qquad 0
\leq \alpha \leq 1.
\] 
Define $f(\alpha) = L(\eta_\alpha)$. It is a continuous function of $\alpha$.
Since $f(0) = L (\eta^\dag)< s$ and $f(1) =L(\tilde\eta)> s$, by the
intermediate value theorem, there exists an
$\breve{\alpha} \in (0,1)$ such that $f(\breve{\alpha}) = s$. 
Denote $\breve{\eta}
= \breve{\alpha} \tilde\eta + (1 - \breve{\alpha}) \eta^\dag$.
Immediately $L(\breve{\eta}) = f(\breve{\alpha})$ = s. 

If \eqref{temp_1} holds, from the convexity of $C(\cdot)$, we have
\begin{equation*}
C(\eta^\dag) <   C(\breve{\eta}) \leq \breve{\alpha} C(\tilde\eta) + (1 - \breve{\alpha})
  C(\eta^\dag),
\end{equation*}
which implies 
\begin{equation}\label{eq:convexity1}
C(\eta^\dag) < C(\tilde \eta).
\end{equation}
On the other hand, we can write 
$\tilde\eta = \eta^\dag + \breve{\beta}(\breve\eta-\eta^\dag)$, where
$\breve{\beta} = \breve{\alpha}^{-1}>1$. 
If \eqref{temp_2} holds, then
\begin{equation}\label{eq:convexity2}
\begin{split}
C(\tilde\eta) - C(\breve\eta) 
& = C(\eta^\dag + \breve\beta (\breve\eta - \eta^\dag)) 
- C(\eta^\dag +  (\breve\eta - \eta^\dag)) \\
& \geq (\breve\beta-1)    \frac{\partial}{\partial \beta}\, C(\eta^\dag + \beta (\breve\eta -
   \eta^\dag) ) \bigg|_{\beta=1^+}>0.
\end{split}
\end{equation}
Both \eqref{eq:convexity1} and \eqref{eq:convexity2} imply that 
$\tilde\eta$ with $L(\tilde{\eta}) > s$ cannot be the minimizer of $C(\cdot)$.
\end{proof}

\begin{proof}[Proof of Theorem \ref{approx_err}] 
We assume $p \leq m+1$ without loss of generality, since we can replace $p$ by $p'
=p \wedge (m+1)$ otherwise. For $\eta^*_n$ as in Proposition \ref{schumaker},
we have that $\|\eta^*_n - \eta_0\| \leq C_1 \delta_n^p$ and $J_q(\eta^*_n) \leq C_3  {\delta_n^{2(p-q) \wedge 0}}$. Therefore
\begin{equation}\label{eq:schumaker}
\|\eta^*_n - \eta_0\| + \lambda_n^{1/2} J_q^{1/2}(\eta^*_n) \leq C_1 \delta_n^p +  C_3^{1/2}\lambda_n^{1/2} {\delta_n^{(p-q)\wedge 0}}. 
\end{equation}
In the following we will repeatedly use the inequality
\begin{equation}\label{eq:sq-ineq}
\frac{1}{2}(u+v)^2 \leq u^2 + v^2 \leq (u+v)^2, \qquad u,v>0. \end{equation}
to bound $(\delta_n^p + \lambda_n^{1/2} {\delta_n^{(p-q)\wedge 0}})^2$ and $ \delta_n^{2p} + \lambda_n  {\delta_n^{2(p-q)\wedge 0}}$ by each other.

We apply the Convexity Lemma (Lemma~\ref{lemma:convex}) to the convex functional 
\[
C(g) = -\Lambda(g) + \lambda_n J_q(g)
\]
and the continuous functional
\[
L(g) = \|g - \eta_n^*\| + \lambda_n^{1/2} J_q^{1/2}(g-\eta_n^*),
\]
both defined on $\mathcal{C} = \mathbb{G}_n$.
The continuity of $L(g)$ follows from the fact that
\[
|L(g_1) - L(g_2) | \leq \|g_1 - g_2\| + \lambda_n^{1/2}
J_q^{1/2}(g_1 - g_2).
\] When applying the lemma, take $s= a (\delta_n^p + \lambda_n^{1/2} {\delta_n^{(p-q)\wedge 0}})$,
where $a>0$ is a constant to be determined later.

Take $\eta^\dag=\eta^*_n$ in Lemma~\ref{lemma:convex}. We have $L(\eta^*_n) =0$.
We will show that 
\begin{equation}\label{eq:comp-C}
C(\eta^*_n) < C(g), \qquad g\in \mathbb{G}_n \text{ with } L(g) = s.
\end{equation}
Then, the Convexity Lemma implies that the minimizer $\bar\eta_n$ of $C(g)$ in $\mathbb{G}_n$
satisfies $L(\bar\eta_n) <s$. Consequently, by the triangle inequality and \eqref{eq:schumaker},
\begin{align*}
\|\bar{\eta}_n - \eta_0\| + \lambda_n^{1/2} J_q^{1/2}(\bar{\eta}_n) 
& \leq L(\bar\eta_n) + \|\eta_n^* - \eta_0\| + \lambda_n^{1/2} J_q^{1/2}(\eta_n^*) \\
& \leq a (\delta_n^p + \lambda_n^{1/2} {\delta_n^{(p-q)\wedge 0}}) + C_1 \delta_n^p + C_3^{1/2}\lambda_n^{1/2} {\delta_n^{(p-q)\wedge 0}}.
\end{align*}
By using \eqref{eq:sq-ineq}, we have that
\begin{equation}\label{bar_0_diff_ext}
 \|\bar{\eta}_n - \eta_0\|^2 + \lambda_n J_q(\bar{\eta}_n) = O(\delta_n^{2p} \vee \lambda_n {\delta_n^{2(p-q)\wedge 0}}),
\end{equation}
which is the desired result.

It remains to show \eqref{eq:comp-C}. 
By Proposition \ref{schumaker}, $\|\eta^*_n - \eta_0\|_\infty \leq C_2 \delta_n^{p-1/2}$.
For $g\in \bbG_n$ with $L(g)\leq s$, we have
\begin{equation}\label{etanear0_2}
\|g - \eta^*_n\|_\infty  \leq A_n \|g - \eta^*_n\| \leq A_n L(g)  \leq  A_n a (\delta_n^p + \lambda_n^{1/2} {\delta_n^{(p-q)\wedge 0}}),
\end{equation}
and therefore,  
\begin{equation}\label{etanear0_3}
\begin{split}
\|g-\eta_0\|_\infty &\leq \|g - \eta^*_n\|_\infty + \|\eta^*_n -\eta_0\|_\infty \\ 
&\leq  A_n a (\delta_n^p + \lambda_n^{1/2}  {\delta_n^{(p-q)\wedge 0}}) + C_3 \delta_n^{p-1/2}= o(1)
\end{split}
\end{equation}
(since $p> 1/2$). Thus, $\|g-\eta_0\|_\infty \leq B$ when $n$ is large, for $B$ in Condition~\ref{cond:exp-loglik}. 
Then, use Condition~\ref{cond:exp-loglik} to obtain
\begin{equation}\label{etanear0_4}
\begin{split}
C(g) + \Lambda(\eta_0)
& = -\Lambda(g) + \Lambda(\eta_0) + \lambda_n J_q(g) \\
& \geq M_1 \|g - \eta_0\|^2 + \lambda_n J_q(g)\\
&\geq \frac{1}{2} (M_1\wedge 1) \{\|g - \eta_0\| + \lambda_n^{1/2} J_q^{1/2}(g)\}^2,\\
\end{split}
\end{equation}
and
\begin{equation}\label{etastarnear0}
\begin{split}
C(\eta^*_n) + \Lambda(\eta_0)
& = -\Lambda(\eta^*_n) + \Lambda(\eta_0) + \lambda_n J_q(\eta^*_n) \\
& \leq M_2 \|\eta^*_n - \eta_0\|^2 + \lambda_n J_q(\eta^*_n)\\
& \leq (M_2\vee 1) \{\|\eta^*_n - \eta_0\| + \lambda_n^{1/2}
J_q^{1/2}(\eta^*_n)\}^2.
\end{split}
\end{equation}
For $g\in \bbG_n$ with $L(g) = s$, by the triangle inequality and \eqref{eq:schumaker}, we have that 
\begin{align*}
a (\delta_n^p + \lambda_n^{1/2} {\delta_n^{(p-q)\wedge 0}}) & = \|g-\eta_n^*\| + \lambda_n^{1/2} J_q^{1/2}(g - \eta_n^*)\\
& \leq \|g-\eta_0\| + \lambda_n^{1/2} J_q^{1/2}(g) +\|\eta_n^*-\eta_0\|   + \lambda_n^{1/2} J_q^{1/2}(\eta_n^*)\\
& \leq \|g-\eta_0\| + \lambda_n^{1/2} J_q^{1/2}(g) + C_1\delta_n^p + C_3^{1/2}\lambda_n^{1/2} {\delta_n^{(p-q)\wedge 0}}.
\end{align*}
Using the above inequality and \eqref{eq:schumaker} we obtain that, by taking $a$ large enough, the right hand
side of (\ref{etanear0_4}) is strictly greater than the right hand side of
(\ref{etastarnear0}). This proves \eqref{eq:comp-C}.

It follows from \eqref{etanear0_3} that, for any $g\in \bbG_n$ with $L(g)\leq s$, we have
\begin{equation}\label{etanear0_3b}
\|g\|_\infty \leq \|g - \eta_0\|_\infty +\|\eta_0\|_\infty 
< M \|\eta_0\|_\infty
\end{equation}
for large $n$. Since $L(\bar\eta_n) <s$, 
\eqref{etanear0_3b} implies that $\|\bar\eta_n\|_\infty \leq M\|\eta_0\|_\infty<\infty$. It follows again from \eqref{etanear0_3} that $\|\bar\eta_n - \eta_0\|_\infty = o(1)$. The proof is complete.
\end{proof}

\begin{proof}[Proof of Theorem~\ref{est_err}]
We apply the Convexity Lemma (Lemma~\ref{lemma:convex}) to the convex functional 
\[
C(g) = -\ell(g) + \lambda_n J_q(g)
\]
and the continuous functional
\[
L(g) = \|g - \bar\eta_n\| + \lambda_n^{1/2} J_q^{1/2}(g-\bar\eta_n),
\]
both defined on $\mathcal{C} = \mathbb{G}_n$.
We take 
\[
s^2 = a^2 \biggl(\frac{1}{n  \lambda_n^{1/(2q)}}\bigwedge \frac{1}{n
  \delta_n} \biggr)
\]
when applying this lemma,
where $a>0$ is a constant to be determined later.

Take $\eta^\dag = \bar\eta_n$ in Lemma~\ref{lemma:convex}. 
We have $L(\bar\eta_n) =0 < s$.
We will show that 
\begin{equation}\label{eq:comp-C2}
\frac{\partial}{\partial\alpha} C(\bar\eta_n + \alpha (g-\bar\eta_n)) \bigg|_{\alpha=1^+}>0, \qquad g\in \mathbb{G}_n \text{ with } L(g) = s.
\end{equation}
Then, the Convexity Lemma implies that the minimizer $\hat\eta_n$ of $C(g)$ in $\mathbb{G}_n$
satisfies $L(\hat\eta_n) \leq s$. 
Hence, 
\begin{equation}\label{bar_0_diff_ext2}
 \|\hat{\eta}_n - \bar\eta_n\|^2 + \lambda_n J_q(\hat{\eta}_n - \bar\eta_n) 
\leq
s^2 = a^2 \biggl(\frac{1}{n  \lambda_n^{1/(2q)}}\bigwedge \frac{1}{n
  \delta_n} \biggr)
\end{equation}
which is the desired result.

It remains to show \eqref{eq:comp-C2}.
Because $J_q(\cdot)$ is a quadratic functional, we have the expansion
\[
J_q(\bar\eta_n + \alpha (g-\bar\eta_n)) = J_q(\bar\eta_n) + 2\alpha J_q(\bar\eta_n, g-\bar\eta_n) + \alpha^2J_q(g-\bar\eta_n).
\]
This together with the definition of $C(\cdot)$ imply that
\[
\frac{\partial}{\partial\alpha} C(\bar\eta_n + \alpha (g-\bar\eta_n))
\bigg|_{\alpha=1^+}
= \mathrm{I} + \mathrm{II}
\]
where (using~\eqref{eq:gradient})
\[
\mathrm{I} = - \frac{d}{d\alpha} \ell(\bar\eta_n +\alpha(g-\bar\eta_n))
\bigg|_{\alpha=0^+} + 2\lambda_n J_q(\bar\eta_n, g-\bar\eta_n)
= - (E_n- E)\; \dot{l}[\bar\eta_n;g-\bar\eta_n],
\]
and
\[
\mathrm{II}=  -\frac{d}{d\alpha} \ell(\bar\eta_n+ \alpha(g-\bar\eta_n)) \bigg|_{\alpha=1^+} 
+ \frac{d}{d\alpha} \ell(\bar\eta_n + \alpha(g-\bar\eta_n)) \bigg|_{\alpha=0^+} 
+ 2 \lambda_n J_q(g-\bar\eta_n).
\]

Now consider $g\in \mathbb{G}_n$ with $L(g) \leq s$. 
By Condition \ref{l_12der} (i),
\begin{equation}\label{1st-der}
\begin{split}
| \mathrm{I} | & = \{\|g-\bar\eta_n\|^2 + \lambda_n J_q(g - \bar{\eta}_n) \}^{1/2}
O_P \biggl(\biggl(\frac{1}{n\lambda_n^{1/(2q)}} \bigwedge
\frac{1}{n\delta_n}\biggr)^{1/2}\biggr)\\
&\leq s \, O_P\biggl(\frac{s}{a}\biggr)  =O_P\biggl(\frac{s^2}{a}\biggr).
\end{split}
\end{equation}
On the other hand, by the definition of $A_n$,
\[
\|g - \bar{\eta}_n\|_\infty \leq A_n \|g - \bar{\eta}_n\| = A_n a
\biggl(\frac{1}{n\lambda_n^{1/(2q)}} \bigwedge
\frac{1}{n\delta_n}\biggr)^{1/2}= o(1).
\]
Thus, $\|g - \bar{\eta}_n\|_\infty \leq B$ for large $n$ where $B$ is
as in Condition \ref{l_12der} (ii). It then follows from this
condition that, for $g\in \bbG_n$ with $L(g)=s$,
\begin{equation}  \label{2nd-der}
\begin{split}
\mathrm{II} & \geq M \|g- \bar\eta_n\|^2 + 2 \lambda_n  J_q(g-\bar\eta_n)\\
& \geq \frac{1}{2} (M\wedge 2) \{\|g- \bar\eta_n\| + \lambda_n^{1/2}
J_q^{1/2}(g-\bar\eta_n)\}^2
= \frac{1}{2} (M\wedge 2) s^2.
\end{split}
\end{equation}
Therefore, by taking a sufficient large $a$,
\[
\mathrm{I} + \mathrm{II} \geq \frac{1}{2} (M\wedge 2) s^2- O_P\biggl(\frac{s^2}{a}\biggr)>0.
\]
Thus we have proved \eqref{eq:comp-C2}. This completes the proof of the theorem.
\end{proof}

\section{Useful lemmas for verifying the conditions of the master theorems}\label{sec:strategy}

This section develops three lemmas that provide sufficient conditions
for Conditions~$\ref{cond:exp-loglik}$ and $\ref{l_12der} (i) (ii)$, respectively.

\begin{lemma}~\label{lemma:exp-loglik}
Suppose $\|h_1\|_\infty \leq C$ for some constant $C>0$. If there are constant $B > 0$ and constants $M_1, M_2 > 0$ such that
 \begin{equation}\label{eq:2ndder-exp-loglik}
    -M_1\|h_2\|^2 \leq \frac{d^2}{d \alpha^2} \Lambda(h_1 +
    \alpha h_2) \leq -M_2 \|h_2\|^2, \quad 0 \leq \alpha \leq 1,
  \end{equation}
whenever $\|h_2\|_\infty \leq B$, then
Condition~$\ref{cond:exp-loglik}$ holds if $\|\eta_0\|\leq C$.
\end{lemma}

This is Lemma A.1 of \cite{huang2001concave}, which is proved easily by a
Taylor expansion at the maximal point of the expected log-likelihood
and noticing that the first order term is zero. As we will show
later in this paper that \eqref{eq:2ndder-exp-loglik}
can be verified easily in various contexts.

\begin{lemma}~\label{lemma:score}
If there exists a constant $M$ such that $\mathrm{Var}\{\dot{l}
[\bar\eta_n;h]\} \leq M $ for any $h$ satisfying $\|h\|^2 =1$, then
Condition~$\ref{l_12der} (i) $ holds.
\end{lemma}

This lemma is a generalization of Lemma~A.2 of
\cite{huang2001concave}, which gives a similar result for polynomial
spline estimators.

\begin{proof}[Proof of Lemma~\ref{lemma:score}.]
Consider an orthonormal basis $\{\psi_k, k=1, \dots, N_n\}$ of
$\bbG_n$.
We have $N_n \asymp \delta_n^{-1}$. Any $g\in \bbG_n$ can
be represented by this basis as $g=\sum_k g_k \psi_k$, where $g_k =
\langle g, \psi_k\rangle$. It follows that
$\dot{l}[\bar\eta_n; g] = \sum_k g_k  \dot{l}[\bar\eta_n; \psi_k].$ 
By the Cauchy--Schwarz inequality and $\|g\|^2 =\sum_k g_k^2$,
\begin{equation}\label{eq:gradient-basis}
\frac{|(E_n-E) \dot{l}[\bar\eta_n; g]|^2}{\|g\|^2+ \lambda_n J_q(g)}
\leq \frac{|(E_n-E) \dot{l}[\bar\eta_n; g]|^2}{\|g\|^2}
\leq  \sum_k \{ (E_n-E) \dot{l}[\bar\eta_n; \psi_k]\}^2
\end{equation}
Since $\|\psi_k\|=1$, by the assumption of the lemma, the expectation of the right hand side
of the above is bounded by $\sum_k \{M/n\} \leq M/(n\delta_n)$.
On the other hand, take the eigen decomposition $g = \sum_\nu g_\nu \phi_\nu$. We
have $\dot{l}[\bar{\eta}_n; g] = \sum_\nu g_\nu\dot{l} [\bar{\eta}_n; \phi_\nu]$.
By the Cauchy--Schwartz inequality and
\[
\|g\|^2 + \lambda_n J_q(g) = \sum_\nu g_\nu^2(1+\lambda_n \rho_\nu),
\]
we have that
\begin{equation} \label{eq:gradient-penalty}
\frac{|(E_n-E) \dot{l}[\bar{\eta}_n; g] |^2}{\|g\|^2 + \lambda_n J_q(g)}
 \leq  \sum_\nu\frac{\{(E_n-E) \dot{l} [\bar{\eta}_n; \phi_\nu] \}^2}{1+\lambda_n \rho_\nu}.
 \end{equation}
Since $\|\phi_\nu\|=1$, by the assumption of this lemma and
Proposition~\ref{gu1}, the expectation of the right hand side of the above is bounded by
\[
\frac{M}n \sum_\nu\frac{1}{1+\lambda_n \rho_\nu} = O\biggl(\frac{1}{n\lambda_n^{1/(2q)}}\biggr).
\]
The conclusion now follows from
\eqref{eq:gradient-basis}--\eqref{eq:gradient-penalty} and the Markov inequality.
\end{proof}

\begin{lemma}~\label{lemma:hessian}
The following provides a sufficient condition for
Condition~$\ref{l_12der} (ii)$:\\
$(i)$ $\|\bar\eta_n\|_\infty =O(1)$;\\
$(ii)$ For $g\in \bbG_n$, $\ell(\bar\eta_n+ \alpha g)$ as a function of $\alpha$ is
twice continuously differentiable; moreover, there are constants $B >
0$ and $M > 0$ such that 
\begin{equation*}
  \frac{d^2}{d\alpha^2}\ell(\bar\eta_n+ \alpha g) \leq -M \|g\|^2,
  \hspace{.5cm} 0 \leq \alpha \leq 1,
\end{equation*}
holds for $g \in \bbG_n$ with $\|g\|_\infty \leq B$, with probability tending to one as $n \to \infty$.
\end{lemma}

When using this lemma, we only need to verify Part (ii) of the
condition, since Part (i) is a consequence of Theorem~\ref{approx_err}.
Part (ii) of the condition has been used in~\cite{huang2001concave}
when studying rates of convergence of polynomial spline estimators.

\begin{proof}[Proof of Lemma~\ref{lemma:hessian}.]
Since
\[
   \frac{d}{d\alpha} \ell(\bar{\eta}_n + \alpha
   g)\bigg|_{\alpha = 1} - \frac{d}{d\alpha} \ell(\bar{\eta}_n + \alpha
   g)\bigg|_{\alpha = 0} 
= \int_0^1   \frac{d^2}{d\alpha^2}\ell(\bar\eta_n + \alpha g) \,
d\alpha,
\]
the result is straightforward.
\end{proof}

\section{Application I: regression}\label{sec:regression}
Consider the problem of estimating the conditional mean function $\eta_0(x) = E (Y| X=x)$ based on an i.i.d.\ sample of $\mathbf{W}= (X, Y)$, denoted as $\mathbf{W}_i = (X_i,Y_i), i=1, \dots, n$. 
For a candidate function $h$ of the unknown function $\eta_0$, define the ``log-likelihood'' functional as
\[
\ell(h; \mathbf{W}_1, \ldots, \mathbf{W}_n) = - \frac{1}{n} \sum_{i=1}^n \{Y_i - h(X_i)\}^2.
\]
This can be interpreted as a (conditional) log-likelihood (up to a scale factor) when the conditional distribution
of $y_i$ given $x_i$ is Gaussian or a pseudo log-likelihood without the distribution assumption.

We verify conditions used in the master theorems under the following primitive assumptions. 
\begin{assumption*}[\textbf{REG}] 

$(i)$  The function $\eta_0$ is bounded on $\mathcal{U}$.

$(ii)$  There is a constant $D>0$ such that $\mathrm{Var}(Y|X=x) \leq D $ for all $x$. 

$(iii)$ The distribution of $X$ is absolutely continuous and its density function is bounded away from zero and infinity on $\mathcal{U}$, that is, there exist constants $C_1, C_2 >0$ such that 
\begin{equation*}
    C_1 \leq f_X(x) \leq C_2, \quad \text{for } x \in \mathcal{U}.
\end{equation*}
\end{assumption*}

The expected log-likelihood is
\[
\Lambda(\eta) = - E [\{Y_i - h(X_i)\}^2].
\]
Define the empirical and theoretical norms as in Section~\ref{sec:splines} with the weight function being $w(x) \equiv 1$. It is easy to see that 
\begin{equation*}
\frac{d^2}{d\alpha^2}\Lambda(h_1+\alpha h_2) = - 2 \|h_2 \|^2,
\end{equation*}
and thus \eqref{eq:2ndder-exp-loglik} holds with $M_1= M_2 = 2$.  Condition~$\ref{cond:exp-loglik}$ then follows from Lemma~\ref{lemma:exp-loglik}.

Note that
\[
\dot{l}[\bar\eta_n;h](\mathbf{W}_1) = \{\bar{\eta}_n(X_1) -  Y_1\} h(X_1).
\]
Since we apply Theorem~\ref{est_err} after we apply Theorem~\ref{approx_err}, we can use the conclusion of Theorem~\ref{approx_err} and assume that $\|\bar\eta_n\|_\infty \leq M$ for some constant $M>0$ when $n$ is large enough. Suppose $\|h\|^2=1$.  
Let $\epsilon_1 = Y_1 - \eta_0(X_1)$. We have that
\begin{align*}
   \text{Var}[\{\bar{\eta}_n(X_1) - Y_1\} h(X_1)] &\leq E[\{\bar{\eta}_n(X_1) - Y_1\}^2 h(X_1)^2] \\
&= E[\{\bar{\eta}_n(X_1) - \eta_0(X_1)\}^2 h(X_1)^2] + E[\epsilon_1^2 h(X_1)^2]\\
& \leq \|\bar{\eta}_n - \eta_0\|_\infty^2 + D
\leq (M+ \|\eta_0\|_\infty)^2 + D,
\end{align*}
which is the condition needed for applying Lemma~\ref{lemma:score}.
Condition~$\ref{l_12der}(i)$ then follows from Lemma~\ref{lemma:score}. 

Finally,
\begin{equation*}
  \frac{d^2}{d\alpha^2}\ell(\bar\eta_n + \alpha\, g; \mathbf{W}_1, \ldots, \mathbf{W}_n) = - \frac{2}{n}\sum_{i = 1}^n  g^2(X_i) = - 2 \|g\|_n^2.
\end{equation*}
Proposition~\ref{equiv_norm} implies that Part (ii) of the condition in Lemma~\ref{lemma:hessian} holds if $\lim_n N_n \log(n) /n = 0$, and thus Condition~\ref{l_12der}($ii$) holds according to this lemma.

Verification of conditions is complete.

\section{Application II: generalized regression}\label{sec:gr}
Our setup of generalized regression follows \citet{stone1986dimensionality, stone1994use} and \citet{huang1998functional}. In a generalized regression model, the conditional distribution of $Y$ given $X$ is characterized by an exponential family of distributions
\begin{equation}
  \label{density_gm}
  P(Y \in dy|X = x) = \exp \{B(\eta_0(x))y - C(\eta_0(x))\} \Psi(dy),
\end{equation}
where $\Psi(\cdot)$ is a nonzero measure on $\mathbb{R}$ that is not concentrated on a single point, and 
$C(\eta) = \log \int_{\mathbb{R}} \exp \{B(\eta) y\}\, \Psi(dy)$ is a well-defined normalizing constant for each $\eta$ in an open subinterval $\mathcal{I}$ of $\mathbb{R}$. Define $A(\eta) = C'(\eta)/B'(\eta)$ if the derivatives exist. The standard theory of exponential family of distributions gives that $E(Y|X=x) = A(\eta_0(x))$. 

The goal is to estimate the unknown function $\eta_0$ based on an $i.i.d.$ sample of $(X,Y)$, denoted as $\mathbf{W}_1=(X_1,Y_1), \ldots, \mathbf{W}_n=(X_n,Y_n)$. The scaled (conditional) log-likelihood at a candidate function $h$ is given by
\begin{equation*}
  \ell(h;\mathbf{W}_1, \ldots, \mathbf{W}_n) = \frac{1}{n}\sum_{i=1}^n \{B(h(X_i))Y_i - C(h(X_i))\},
\end{equation*}
and its expectation is
\begin{equation*}
  \Lambda(h) = E\{B(h(X_1))A(\eta_0(X_1)) - C(h(X_1))\}.
\end{equation*}
Define the empirical and theoretical norms as in Section~\ref{sec:splines} with the weight function being $w(x) \equiv 1$.

We verify conditions used in the master theorems under the following primitive assumptions.

\begin{assumption*}[\textbf{GR}]\label{con_gr}

$(i)$ $B(\cdot)$ is twice continuously differentiable and its first derivative is strictly positive on $\mathcal{I}$. 

$(ii)$ There is a subinterval $S$ of $\bbR$ such that $\Psi$ is concentrated on $S$ and
\begin{equation}\label{eq:gr1}
    B^{''}(\xi) \, y - C^{''}(\xi) < 0, \qquad y \in \mathring{S}, \xi \in \mathcal{I}
\end{equation}
where $\mathring{S}$ is the interior of $S$. If $S$ is bounded, \eqref{eq:gr1} holds for at least one of its endpoints.

$(iii)$ $P(Y\in S) =1$ and $E(Y|X=x) = A(\eta_0(x))$ for $x \in \mathcal{U}$.

$(iv)$ There is a compact subinterval $\mathcal{K}_0$ of $\mathcal{I}$ such that $\mathrm{range}(\eta_0)\subset \mathcal{K}_0$.

$(v)$ There is a constant $D>0$ such that $\mathrm{Var}(Y|X=x) \leq D $ for all $x$. 

$(vi)$ The distribution of $X$ is absolutely continuous and its density function is bounded away from zero and infinity on $\mathcal{U}$, that is, there exist constants $C_1, C_2 >0$ such that 
\[
    C_1 \leq f_X(x) \leq C_2, \quad \text{for } x \in \mathcal{U}. 
\]
\end{assumption*}

The same set of assumptions was used is \cite{huang1998functional}, where one can find more detailed discussions.  In particular, Assumptions \textbf{GR}(i)(ii) are satisfied by many familiar exponential families of distributions, including Normal, Binomial-probit, Binomial-logit, Poisson, gamma, geometric and negative binomial distribution; see Stone (1986). By relaxing the restriction that $\mathcal{I}=\mathbb{R}$, the identity link is allowed for Poisson regression and Binomial regression. It is important to point out that using this set of assumptions, the conditional distribution of $Y$ given $X=x$ does not necessarily belong to the exponential family, we only need that the conditional mean of $Y$ given $X=x$ is $A(\eta_0(x))$, as stated in \textbf{GR}(iii). As explained in \cite{huang1998functional}, this means that $\eta_0(\cdot)$ maximizes the expected log-likelihood functional $\Lambda(h)$. 

Luckily, \cite{huang1998functional} has already verified for us all the conditions used in our master theorems under the above assumptions. In particular, Lemma~4.1 of \cite{huang1998functional} verified Condition~\ref{cond:exp-loglik}; Proof of Claim 2 given on page 68 of \cite{huang1998functional} verified the condition in our Lemma~\ref{lemma:score} and thus verified Condition~$\ref{l_12der}(i)$; Lemma~4.3 of \cite{huang1998functional} verified Part (ii) of the condition in our Lemma~\ref{lemma:hessian} and thus verified Condition~$\ref{l_12der}(ii)$.

\section{Application III: probability density estimation}\label{prob_density}
Suppose $X$ is a random variable defined on a bounded interval $\mathcal{U}$ and has a density function $f_0(x)$. The goal is to estimate the unknown function $f_0(x)$ based on an i.i.d.\ sample of $X$, denoted as $X_i, i = 1, \ldots, n$. One difficulty for density estimation using penalized splines is that the density estimator has to satisfy two intrinsic constraints that $f_0$ satisfies, namely, the positivity constraint that $f_0 \geq 0$ and the unity constraint that $\int_{\mathcal{U}}f_0(x) \, dx = 1$. 
Assuming $f_0(x) > 0$ on $\mathcal{U}$, by making the transform
$f_0(\cdot) = \exp \eta_0(\cdot)/\int_{\mathcal{U}}\exp \eta_0(x) \,
dx$ we convert the problem to the estimation of $\eta_0$, which is
free of the two constraints on $f_0$. However, this transformation
creates an identifiability problem, that is, $\eta_0 +c$ and $\eta_0$
give the same density function for any constant $c$. To fix this problem, we require that $\int_{\mathcal{U}} \eta_0(x)\, dx=0$, which ensures a one-to-one correspondence between $f_0$ and $\eta_0$. To define a penalized spline estimator of $\eta_0$, we need to slightly modify our framework by restricting our attention to a subspace of $\bbG_n$,  $\bbG_{n1} = \{g \in \bbG_n: \int_{\mathcal{U}} g(x) \, dx = 0\}.$

We have a concave extended linear model with $\mathbf{W} = X$. The scaled log-likelihood at a candidate function $h$ based on the sampled data is 
\begin{equation*}
\ell(h;\mathbf{W}_1, \ldots, \mathbf{W}_n) = \frac{1}{n} \sum_{i=1}^n  \biggl(h(x_i) -
  \log \int_{\mathcal{U}}\exp h(x) \, dx \biggr). 
\end{equation*}
The expected log-likelihood is 
\begin{equation*}
  \Lambda(h) = E\{h(X)\} - \log \int_{\mathcal{U}}\exp h(x) \, dx.
\end{equation*}

We verify conditions used in the master theorems under the following primitive assumptions.
We make the additional assumption $\int_{\mathcal{U}}h(x)\, dx =0$ when verifying Condition~\ref{cond:exp-loglik} and replace $\mathbb{G}_n$ by $\mathbb{G}_{n1}$ when verifying Condition~\ref{l_12der}.

\begin{assumption*}[\textbf{DEN}]\label{con_den}
The density function $f$ is bounded away from zero and infinity on $\mathcal{U}$, or equivalently, $\eta_0$ is bounded on $\mathcal{U}$.
\end{assumption*}

Let $U$ be a random variable that has a uniform distribution on $\mathcal{U}$. Under the above assumption, we have that, for $h$ satisfying $\int_{\mathcal{U}}h(x)\, dx =0$,
\begin{equation}\label{eq:den-equiv1}
\begin{split}
E\{h^2(U)\} 
= \mathrm{Var}\{h(U)\} & = \inf_c E[\{h(U)-c\}^2]\\
& \asymp \inf_c E_{\eta_0}[\{h(X)-c\}^2] = \mathrm{Var}_{\eta_0}\{h(X)\},
\end{split}
\end{equation}
where the subscript $\eta_0$ emphasizes the fact that the distribution of $X$ is determined by $\eta_0$.
Therefore,
\begin{equation}\label{eq:den-equiv2}
E_{\eta_0}\{h^2(X)\} \lesssim E\{h^2(U)\} \asymp \mathrm{Var}_{\eta_0}\{h(X)\} \leq E_{\eta_0}\{h^2(X)\}.
\end{equation}

Define the empirical and theoretical norms as in Section~\ref{sec:splines} with the weight function being $w(x) \equiv 1$.
Under Assumption~(\textbf{DEN}), the theoretical norm $\|h\| $ is equivalent to $\|h\|_2$, the $L_2$ norm with respect to the Lebesgue measure. It is easy to see that 
\begin{equation*}
\frac{d^2}{d\alpha^2}\Lambda(h_1+\alpha h_2) = -  \mathrm{Var}\{h_2(X_\alpha)\},
\end{equation*}
where $X_\alpha$ has the density $f_{X_\alpha}(x) =
\exp h_\alpha(x)/\int_{\mathcal{U}}\exp h_\alpha(x) \, dx$ and
$h_\alpha = h_1 + \alpha \, h_2 $, $0\leq \alpha\leq 1$. 
For $B, C>0$, if $\|h_1\|_\infty \leq C$, $\|h_2\|_\infty \leq B$, then $\|h_\alpha\|_\infty \leq B+C$ and therefore there are constants $M_1, M_2 >0$ such that $M_2/|\mathcal{U}| |\leq f_{X_\alpha}(x) \leq M_1/|\mathcal{U}|$.
Using the same argument for proving \eqref{eq:den-equiv1}, we obtain that 
\[
M_2\mathrm{Var}\{h_2(U) \}  \leq \mathrm{Var}\{h_2(X_\alpha)\} 
\leq M_1\mathrm{Var}\{h_2(U)\}, 
\]
where $U$ has a uniform distribution on $\mathcal{U}$. Since $\mathrm{Var}\{h_2(U)\}$ is equivalent to $\|h_2\|_2^2$ and also $\|h_2\|^2$ when $h_2$ satisfies $\int_{\mathcal{U}} h_2(x)\,dx =0$, \eqref{eq:2ndder-exp-loglik} holds.  Condition~$\ref{cond:exp-loglik}$ then follows from Lemma~\ref{lemma:exp-loglik}.

To verify Condition \ref{l_12der} ($i$), note that
\begin{equation}\label{1st-der_1_den}
 \dot{l}[\bar\eta_n; h](\mbf{W}_1) = h(X_1) -  E_{\bar\eta_n} \{h(X)\},
\end{equation}
where the subscript $\bar\eta_n$ indicates that the expectation is taken as if the distribution of $X$ is determined by $\bar\eta_n$. It follows that
\[
\text{Var}\{\dot{l}[\bar\eta_n; h](\mbf{W}_1)\}= \text{Var}\{h(X_1)\} \leq \|h\|^2,
\]
indicating that the condition in our Lemma~\ref{lemma:score}
holds. (The restriction $\int_{\mathcal{U}}h(x)\, dx =0$ is taken care of by noticing that the constant function is the eigenfunction corresponds to the zero eigenvalue.) Condition~$\ref{l_12der}(i)$ follows from Lemma~\ref{lemma:score}.

Finally, because
\begin{equation}
  \label{2nd-der_1_den}
  \frac{d^2}{d\alpha^2} \ell(\bar\eta_n + \alpha g; \mathbf{W}_1, \ldots, \mathbf{W}_n) = 
\frac{d^2}{d\alpha^2}\Lambda(\bar\eta_n +\alpha g),
\end{equation}
the verification of Condition~$\ref{cond:exp-loglik}$ implies Part (ii) of the condition in our Lemma~\ref{lemma:hessian},
which in turn implies Condition~$\ref{l_12der}(ii)$ using the lemma.

\section{Application IV: counting process regression}\label{sec:countingp}

The counting process regression provides a general framework for
survival analysis with censored data \citep{andersen1993}. Here we
adopt the setup used in Section~3 of \cite{huang2001concave}. Let
$\mathcal{T} = [0,\tau]$ for some $\tau>0$. Suppose
$(\Omega,\mathcal{F},P)$ is a complete probability space and
$\{\mathcal{F}_t : t \in \mathcal{T}\}$ is a filtration satisfying the
``usual conditions,'' that is, $\mathcal{F}_t\subset \mathcal{F}$ is a
family of right-continuous, increasing $\sigma$-algebras and
$\mathcal{F}_0$ contains the $P$-null sets of $\mathcal{F}$. Let
$\{N(t): t \in \mathcal{T} \}$ be an adapted \citep{andersen1993} counting process with intensity
\begin{equation}\label{intensity_cp}
E[N(dt)|\mathcal{F}_{t-}] = Y(t) \exp\eta_0({X}(t))\,dt,
\end{equation}
where $Y(t)$ is a $\{0,1\}$-valued, predictable process, indicating
the times at which the process $N(t)$ is under observation, and $X(t)$
is an $\mathcal{U}$-valued, predictable covariate process. Our goal
is to estimate the log-hazard function $\eta_0$ based on an i.i.d.\
sample of $\mathbf{W} =  \{(N(t), Y(t), X(t)): t\in \mathcal{T}\}$, denoted as
$\mathbf{W}_i = \{(N_i(t), Y_i(t), X_i(t)): t\in \mathcal{T}\}, 1 \leq i \leq n$. 

The marker dependent hazard model~\citep{nielsen1995kernel} of hazard regression with right-censored survival data is a special case of this setup.  Specifically, one observes
$(T \wedge C, I(T\leq C))$,  where $T$ is the survival time of a subject and $C$ is the censoring time. (To avoid notational confusion, we do not use $C$ to denote a constant throughout this section.) Suppose $T$ and $C$ are conditional independent given the process $X(t)$, and the conditional hazard of $T$ given 
$\{X(s), s\leq t\}$ is $\exp\eta_0({X}(t))$. Let $N(t) = I(T \leq C \wedge t)$ be the counting process with a single jump at the survival time $T$ if uncensored. Then $N(t)$ has the intensity given by \eqref{intensity_cp}, with $Y(t) = I(T\wedge C \geq t)$ being the indicator that the subject is observed to be at risk at time $t$. 

This is a concave extended linear model. The scaled log-likelihood for a candidate function $h$ of $\eta_0$ is
\begin{equation*}
  \ell(h; \mathbf{W}_1, \ldots, \mathbf{W}_n) = \frac{1}{n} \sum_{i=1}^n
  \biggl(\int_{\mathcal{T}}h({X}_i(t)) N_i(dt) -
  \int_{\mathcal{T}}Y_i(t)\exp h({X}_i(t))\, dt\biggr).
\end{equation*}
The expected log-likelihood is 
\begin{equation*}
  \Lambda(h) = E\biggl(\int_{\mathcal{T}}h({X}(t))N(dt) -
  \int_{\mathcal{T}}Y(t)\exp h({X}(t))\,dt\biggr).
\end{equation*}
For the marker dependent hazard model, the above log-likelihood reduces to the usual form
\begin{equation*}
  \ell(h) = \frac{1}{n} \sum_i \biggl(h(X(T_i)) I\{T_i \leq C\} - \int_0^{T_i \wedge C} \exp h(X_i(t))dt \biggr),
\end{equation*}
and similarly for the expected log-likelihood.

We verify conditions used in the master theorems under the following primitive assumptions.

\begin{assumption*}[\textbf{CP}]\label{con_cp}
$(i)$  The function $\eta_0$ is bounded on $\mathcal{U}$.

$(ii)$  For fixed $t \in \mathcal{T}$, the Radon-Nikodym derivative of the measure $P(Y(t)=1, X(t) \in
  \cdot)$ w.r.t.\ the Lebesgue measure on $\mathcal{U}$ exists and is denoted as
$f_{Y(t)=1,X(t)}(t,x)$. As a function of $(t,x)$, $f_{Y(t)=1,X(t)}(t,x)$ is bounded away from 0 and infinity uniformly in $t \in \mathcal{T}$ and $x \in \mathcal{U}$. 
\end{assumption*}

Define the empirical inner product and corresponding squared norm by 
\[
\langle h_1,h_2 \rangle_n = E_n \int_{\mathcal{T}} Y(t)h_1(X(t)) \allowbreak h_2(X(t)) \, dt
\] 
and $\|h_1\|_n^2 = \langle h_1,h_1 \rangle_n$.
Define the theoretical inner product and the corresponding squared norm 
by
\[
\langle h_1,h_2 \rangle = E\int_{\mathcal{T}} Y(t)h_1({X}(t))h_2({X}(t))) \,dt
\]
and $\|h_1\|^2 = \langle h_1,h_1 \rangle$.
Under Assumption~(\textbf{CP})(ii), the theoretical inner and norm have the forms generally given 
in Section~\ref{sec:splines} with a specific weight function that is bounded away from 0 and infinity. In fact,
\[
\langle h_1,h_2 \rangle = \int_{\mathcal{U}} h_1(x)h_2(x) w_{cp}(x)\, dx.
\] 
for $w_{cp}(x) = \int_{\mathcal{T}} f_{Y(t)=1,{X}(t)}(t,{x}) \,dt$.  The corresponding theoretical norm $\|h\| $ is equivalent to $\|h\|_2$, the $L_2$-norm w.r.t.\ the Lebesgue measure.
Under Assumption~(\textbf{CP})(ii), it is easy to see that 
\begin{equation*}
\begin{split}
\frac{d^2}{d\alpha^2}\Lambda(h_1+\alpha h_2) 
& = -  E\biggl( \int_{\mathcal{T}}Y_i(t) h^2_2(X_i(t)) \exp h_1(X_i(t))\, dt\biggr)\\
& = - \int_{\mathcal{T}} h_2^2(x) \exp h_1(x) w_{cp}(x)\, dx.
\end{split}
\end{equation*}
If $\|h_1\|_\infty \leq C$, the above quantity is bounded above and below by a constant multiple of $\|h_2\|^2_2$, and also of $\|h_2\|^2$.  This indicates that \eqref{eq:2ndder-exp-loglik} holds.  Condition~$\ref{cond:exp-loglik}$ then follows from Lemma~\ref{lemma:exp-loglik}.

Note that
\begin{equation*}  
\dot{l} [\bar\eta_n; h] (\mbf{W}_1) = \int_{\mathcal{T}} h(X_1(t))N_1(dt) -
  \int_{\mathcal{T}} Y_1(t) \exp[\bar{\eta}_n\{X_1(t)\}] h(X_1(t)) \,dt.
\end{equation*}
Appendix B of \cite{huang2001concave} showed that
\[
\mathrm{Var} \biggl( \int_{\mathcal{T}} h(X_1(t)) N_1(dt)\biggr) \leq M_1 \|h\|^2.
\]
Moreover, if $\|\bar\eta_n\|_\infty \leq M_2$,
\begin{align*}
& \mathrm{Var}  \biggl(\int_{\mathcal{T}} Y_1(t) \exp[\bar{\eta}_n\{X_1(t)\}] h(X_1(t))\, dt\biggr)\\
& \qquad \leq  |\mathcal{T}|\exp (2M_2) E \biggl(\int_{\mathcal{T}} Y_1(t) h^2(X_1(t))\, dt \biggr)
=  |\mathcal{T}| \exp (2M_2) \|h\|^2.
\end{align*}
The above two displayed inequalities together imply the condition in our Lemma~\ref{lemma:score} and thus 
Condition~$\ref{l_12der}(i)$ follows from the lemma.

Finally, if $\|\bar\eta_n\|_\infty \leq C$,
\begin{equation*}
\begin{split}
\frac{d^2}{d\alpha^2}\ell(\bar\eta+\alpha g) 
& = -  \frac{1}{n} \sum_{i=1}^n  \biggl( \int_{\mathcal{T}}Y_i(t) g^2(X_i(t)) \exp \bar\eta_1(X_i(t))\, dt\biggr)\\
& \leq - \exp (-C) \frac{1}{n} \sum_{i=1}^n  \biggl( \int_{\mathcal{T}}Y_i(t) g^2(X_i(t)) \, dt\biggr).
\end{split}
\end{equation*}
It follows from equivalence of the empirical and theoretical norms that Part (ii) of the condition in Lemma~\ref{lemma:hessian} holds, and thus Condition~\ref{l_12der}($ii$) holds according to the lemma.

\section{Application V: quantile regression}\label{sec:quantile}

Fixing $\tau \in (0,1)$, let $\eta_0(x)$ be the $\tau$-th quantile of the conditional distribution
of $Y|X=x$. We would like to estimate $\eta_0$ based on an i.i.d.\
sample of $\mathbf{W}= (X, Y)$, denoted as $\mathbf{W}_i = (X_i,Y_i), i=1, \dots, n$. 
For a candidate function $h$ of the unknown function $\eta_0$, define the ``log-likelihood'' functional as
\[
\ell(h; \mathbf{W}_1, \ldots, \mathbf{W}_n) = - \frac{1}{n}
\sum_{i=1}^n \rho_\tau(Y_i - h(X_i)),
\]
where $\rho_\tau(u) = (\tau - \mathbf{1}_{(u < 0)}) u$ is the check
function for quantile at the level $\tau$. This can be interpreted as a pseudo log-likelihood without making a
distribution assumption on the conditional distribution of $Y$ given
$X$. The quantile function $\eta_0$ maximizes the expected log-likelihood functional
\begin{equation}
  \label{exp_llh_quant_reg}
  \Lambda(h) = - E\{\rho_\tau(Y_i - h(X_i))\}.
\end{equation}

We verify conditions used in the master theorems under the following primitive assumptions. 
\begin{assumption*}[\textbf{QR}]
$(i)$  The function $\eta_0$ is bounded on $\mathcal{U}$.

$(ii)$  There are constants $B>0$ and $M_1, M_2>0$ such that for any interval $A \subset [-B, B]$, 
\[
M_1 |A| \leq P( Y-\eta_0(x) \in A |X=x) \leq M_2 |A|,
\]
where $|A|$ denotes the length of interval $A$.

$(iii)$ The distribution of $X$ is absolutely continuous and its density function is bounded away from zero and infinity on $\mathcal{U}$, that is, there exist constants $C_1, C_2 >0$ such that 
\begin{equation*}
    C_1 \leq f_X(x) \leq C_2, \quad \text{for } x \in \mathcal{U}.
\end{equation*}
\end{assumption*}

Similar to the regression case, define the empirical and theoretical
norms as in Section~\ref{sec:splines} with the weight function being
$w(x) \equiv 1$. Using the Knight identity \citep{knight1998},
\begin{equation}
  \label{Knight}
  \rho_\tau(u - v) - \rho_\tau(u) = v \{\mathbf{1}_{(u \leq 0)} - \tau\}
  + \int_0^v \{\mathbf{1}_{(u \leq s)} - \mathbf{1}_{(u \leq 0)}\} ds,
\end{equation}
we obtain
\begin{align*}
\Lambda(\eta_0 + h) - \Lambda(\eta_0) 
& = - E\{\rho_\tau(Y - \eta_0(X) - h(X)) - \rho_\tau(Y- \eta_0(X))\}\\
& = - E [ h(X) \{\mathbf{1}_{(Y-\eta_0(X) \leq 0)} - \tau\} \\
& \qquad + \int_0^{h(X)} \{\mathbf{1}_{(Y-\eta_0(X) \leq s)} - \mathbf{1}_{(Y-\eta_0(X) \leq 0)} \}\,  ds].
\end{align*}
Note the first part of the expectation is zero by the definition of
$\eta_0$. By conditioning and then changing the order of integration, we have
\begin{align*}
& \Lambda(\eta_0 + h) - \Lambda(\eta_0) \\
& \qquad= - E \biggl[ \int_0^{h(X)} E\{\mathbf{1}_{(Y-\eta_0(X) \leq s)} -
  \mathbf{1}_{(Y-\eta_0(X) \leq 0)}|X \}\,  ds\biggr] \\
& \qquad = - E \biggl[\int_0^{h(X)} \mathrm{sgn}(s) P \{Y- \eta_0(X) \text{ is between 0 and s}|X\}\, ds\biggr].
\end{align*}
If $\|h\|_\infty \leq B$, by Assumption~\textbf{QR}(ii), the above
quantity is between $-M_2\|h\|^2/2$ and $- M_1 \|h\|^2/2$.
This verifies Condition~$\ref{cond:exp-loglik}$ .

Define $\psi(u) =\tau-1$ for $u < 0$, and $\psi(u) = \tau$ for $u
\geq 0$. Then $\psi(u)$ is the derivative of $\rho_\tau(u)$ when $u\not
=0$ and the right derivative when $u=0$. The directional derivative 
at $\bar\eta_n$ along the direction of $g$ is
\[
\dot{l} [\bar\eta; g] (\mathbf{W}_1 ) = g(X_1) \psi(Y_1 - \bar\eta(X_1)).
\]
Since $|\psi(u)| \leq 1$, $\mathrm{Var} \{\dot{l} [\bar\eta; g] (\mathbf{W}_1) \} \leq \|g\|^2$.
Condition~$\ref{l_12der}(i)$ then follows from Lemma~\ref{lemma:score}.

It remains to verify Condition~$\ref{l_12der}(ii)$. Note that
\begin{equation}\label{eq:qr-score1}
\begin{split}
& \frac{d}{d\alpha} \ell(\bar\eta_n + \alpha g) \bigg|_{\alpha=1^+}
- \frac{d}{d\alpha} \ell(\bar\eta_n + \alpha g) \bigg|_{\alpha=0^+}\\
& \qquad =  \frac{1}{n} \sum_{i=1}^n g(X_i) \{\psi(Y_i - \bar\eta_n(X_i) -g(X_i)) -
\psi(Y_i -\bar\eta_n(X_i))\}. 
\end{split}
\end{equation}
Let $\epsilon_i = Y_i - \eta_0(X_i)$. Then
$Y_i - \bar\eta_n(X_i)  = \epsilon_i - \{\bar\eta_n(X_i) - \eta_0(X_i)\}$.
By the definition of $\psi(\cdot)$, the difference $\psi(Y_i - \bar\eta_n(X_i) -g(X_i)) -
\psi(Y_i -\bar\eta_n(X_i))$ is non-zero only when zero is between 
\[
Y_i - \bar\eta_n(X_i) -g(X_i) = \epsilon_i - \{\bar\eta_n(X_i) -
\eta_0(X_i)\} - g(X_i)
\]
and 
\[
Y_i - \bar\eta_n(X_i)  = \epsilon_i - \{\bar\eta_n(X_i) - \eta_0(X_i)\},
\]
or equivalently, when $\epsilon_i$ is between $\bar\eta_n(X_i) -
\eta_0(X_i) - g(X_i)$ and $\bar\eta_n(X_i) - \eta_0(X_i)$, and the value is 
$-\mathrm{sgn}\{g(X_i)\}$. Therefore,
\begin{equation}\label{eq:qr-score2}
 -  \frac{d}{d\alpha} \ell(\bar\eta_n + \alpha g) \bigg|_{\alpha=1^+}
+ \frac{d}{d\alpha} \ell(\bar\eta_n + \alpha g) \bigg|_{\alpha=0^+}
 =  \frac{1}{n} \sum_{i=1}^n |g(X_i)|\, I_i
\end{equation}
where
\[
I_i = \mathrm{I}(\text{$\epsilon_i$ is between $\bar\eta_n(X_i) -
\eta_0(X_i) - g(X_i)$ and $\bar\eta_n(X_i) - \eta_0(X_i)$}).
\]

Applying the Hoeffding inequality, we obtain
\[
P\biggl(\biggl|\frac{1}{n} \sum_{i=1}^n |g(X_i)| \{I_i - E(I_i|X_i)\}
\biggr|\geq t\bigg|X_i, i=1,\dots, n \biggr)
\leq 2 \exp \biggl(-\frac{nt^2}{2\|g\|_n^2}\biggr).
\]
It follows that
\begin{equation}\label{eq:gr-hoeffding}
\frac{1}{n} \sum_{i=1}^n |g(X_i)| I_i - \frac{1}{n} \sum_{i=1}^n |g(X_i)| E(I_i|X_i) 
= \|g\|_n\, o\biggl(\sqrt{\frac{\log n}{n}}\biggr).
\end{equation}
We may focus on $g\in \bbG$ satisfying
$\|g\|_\infty \leq B/2$, where $B$ is the constant in
Assumption~\textbf{QR}(ii).  Since $\|\bar\eta_n - \eta_0\|_\infty =o(1)$,
both $\bar\eta_n(X_i) - \eta_0(X_i) - g(X_i)$ and $\bar\eta_n(X_i) -
\eta_0(X_i)$ are in the interval $[-B,B]$. Using the assumption, we 
have that $P(I_i|X_i) \geq M_1 |g(X_i)|$. Thus,
\begin{equation}\label{eq:gr-cond3.2ii}
\frac{1}{n} \sum_{i=1}^n |g(X_i)| E(I_i|X_i) \geq M_1 \|g\|_n^2. 
\end{equation}
Combining \eqref{eq:qr-score2}--\eqref{eq:gr-cond3.2ii} and using the
equivalence between the empirical and theoretical norms (i.e.,
Proposition~\ref{equiv_norm}), we obtain the desired validity of Condition~$\ref{l_12der}(ii)$.

%
%

\section*{Acknowledgements}
The authors would like to thank the anonymous referees and an
Associate Editor for their constructive comments that significantly improved the quality of this paper.

%
%
 


\bibliographystyle{imsart-nameyear}  


\newpage
\setcounter{page}{1}
\startlocaldefs
\newenvironment{rcases}
  {\left.\begin{aligned}}
  {\end{aligned}\right\rbrace}
\endlocaldefs

\begin{frontmatter}
\title{Supplementary Material to \\Asymptotic Properties of Penalized Spline Estimators in Concave
  Extended Linear Models: Rates of Convergence}
\runtitle{Rates of Convergence of Penalized Splines}

\begin{aug}
\author[A]{\fnms{Jianhua Z.} \snm{Huang} \thanksref{T1}
   \ead[label=e1]{jianhua@stat.tamu.edu}},
 \author[B]{\fnms{Ya} \snm{Su} \ead[label=e2]{suyaf@vcu.edu}}
\thankstext{T1}{Corresponding Author}
\address[A]{Department of Statistics, Texas A\&M University, College Station, TX 77843-3143,
\printead{e1}}
\address[B]{Department of Statistical Sciences and Operations
  Research, Virginia Commonwealth University, Richmond, VA 23284-3083, 
\printead{e2}}
\end{aug}

\begin{abstract}
This main paper develops a general theory on rates of convergence of
penalized spline estimators for function estimation
when the likelihood functional is concave in candidate functions, where the likelihood is
interpreted in a broad sense that includes conditional likelihood,
quasi-likelihood, and pseudo-likelihood. The general theory is applicable to obtain results
in a variety of contexts. This supplementary document contains three
main topics: i. a literarture review of related asymptotic theory for smoothing splines
and polynomial splines; ii. additional examples to illustrate the
application of the general theory; iii. extension of the general
theory in the main paper to penalized tensor product splines and
penalized bivariate splines on triangulations.
\end{abstract}

\begin{keyword}[class=MSC2020]
\kwd[Primary ]{62G20}
\kwd[; secondary ]{62G05, 62G07, 62G08}
\end{keyword}

\begin{keyword}
\kwd{basis expansion, multivariate splines, nonparametric regression, polynomial splines, smoothing splines}
\end{keyword}

\end{frontmatter}


\setcounter{figure}{0}
\setcounter{equation}{0}
\setcounter{table}{0}
\setcounter{section}{0}
\renewcommand{\thefigure}{S.\arabic{figure}}
\renewcommand{\theequation}{S.\arabic{equation}}
\renewcommand{\thesection}{S.\arabic{section}}
\renewcommand{\thesubsection}{S.\arabic{section}.\arabic{subsection}}
\renewcommand{\thetable}{S.\arabic{table}}

In this supplementary document, we first provide in Section~\ref{sec:liter} the related
literature on asymptotic theory of smoothing splines and polynomial
splines. Then we present two additional applications of our
theory in the main paper, namely,
estimation of the drift coefficient of a
diffusion type process in Section~\ref{sec:diffusion}, and estimation of the spectral density function
of a stationary time series in Section~\ref{sec:sden}. To show that the conclusions in our main
theorems hold in these contexts, we need only to verify the conditions
in these theorems. Finally, we present extensions of the theory in the
main paper to two multi-dimensional scenarios, namely, penalized tensor product splines in
Section~\ref{sec:ts}, and penalized bivariate splines on triangulations in Section~\ref{sec:bivariate-splines}.
If not stated otherwise, the numbers of theorems, lemmas, conditions,
equations refer to the main paper.

\section{Literature on asymptotic theory for smoothing
  splines and polynomial splines}~\label{sec:liter}

Since the smoothing spline estimators and polynomial spline estimators
can be considered as two extreme cases of the penalized spline
estimators, it is natural to expect that the asymptotic behaviors of
the penalized spline estimators should be related to these two kinds
of estimators. This section provides the relevant literature
that supplements the literature on penalized splines reviewed in
the main paper.

The asymptotic properties of the integrated mean squared error for
smoothing spines in the context of penalized least squares regression 
have been studied by many authors, including
\citet{craven1978smoothing}, \citet{rice1981integrated}, \citet{rice1983smoothing},
\citet{speckman1985spline}, \citet{cox1988approximation}, \citet{oehlert1992relaxed}. 
The asymptotic rates of convergence for smoothing splines in other contexts 
have also been obtained. For example, \citet{cox1990asymptotic} provided a general 
asymptotic analysis of penalized likelihood estimates,
\citet{silverman1982estimation} and \citet{gu1993smoothing} considered the density estimation,
and \citet{gu1996penalized} considered hazard estimation.
Chapter 9 of \citet{gu2013smoothing} presents a comprehensive treatment of rates of
convergence for smoothing spline estimators in the general framework
of smoothing splines ANOVA. \citet{silverman1984spline} established
the asymptotic equivalence of smoothing spline estimators to certain
kernel estimators by constructing asymptotic equivalent kernels. 
\citet{nychka1995splines} studied the local asymptotic properties of
smoothing spline estimators.
\citet{shang2013local} is a recent rather thorough treatment of
the asymptotic properties of smoothing splines with applications to inference.

The integrated mean squared error for polynomial spline estimators
in least squares regression have been studied by \citet{barrow1978asymptotic},
\citet{agarwal1980asymptotic}, \citet{huang1998projection}, \citet{huang2003asymptotics}. The asymptotic rates
of convergence for polynomials splines have been studied in various 
estimation contexts, usually under a more general setup of structured multivariate
function estimation. For example,
\citet{stone1986dimensionality} and \citet{huang1998functional} considered generalized regression,
\citet{stone1990large} and \citet{huang2001concave} considered density estimation,
\citet{kooperberg1995rate} considered spectral density estimation,
\citet{kooperberg1995l2} considered hazard regression with censored data.
\citet{stone1994use}, \citet{hansen1994extended}, \citet{huang2001concave} presented
theoretical syntheses of rates of convergence for polynomial splines.
\citet{zhou1998local} and \cite{huang2003local} studied the local
asymptotic properties of polynomial spline estimators in nonparametric
regression. 

\section{Application VI: estimation of drift coefficient of diffusion type process}~\label{sec:diffusion}

Diffusion type processes are widely used to describe continuous time stochastic processes with application to physical, biological, medical, economic, and social sciences \citep{rao1999statistical}. As in \cite{stone2003statistical}, we consider nonparametric estimation of the drift coefficient of such a process as a function of some time-dependent covariate while assuming the diffusion coefficient as a function of time is known. 
To be specific, we will define a one-dimensional diffusion type process $Y(t)$ accompanied by a covariate process $X(t)$ 
as 
\begin{equation*}
  dY(t) = \eta_0(X(t)) dt + \sigma(t) dW(t), \quad 0 \leq t \leq \tau,
\end{equation*}
where $0 < \tau < \infty$ and $W(t)$ is a Wiener process. The diffusion coefficient $\sigma(t)$ is a
known function of time, while
the drift coefficient $\eta_0(X(t))$ is an unknown function of the
covariate process $X(t)$. Moreover, let $Z(t)$ be a $\{0,1\}$-valued process as a
censoring indicator, $Z(t) = 1$ if the processes $X(t)$ and $Y(t)$ are observed, and $Z(t) = 0$ otherwise.  

The estimation of $\eta_0$ will be based on a random sample of $n$ realizations of $\mathbf{W}= \{(X(t),Y(t), Z(t)): 0 \leq t \leq \tau\}$, denoted as $\{(X_i(t),Y_i(t), Z_i(t)): 0 \leq t \leq \tau\}, 1 \leq i \leq n$. The scaled (partial) log-likelihood at a candidate function $h$ can be expressed as 
\begin{equation*}
  \begin{split}
&     \ell(h; \mathbf{W}_1, \dots, \mathbf{W}_n) \\
    & \qquad = \frac{1}{n} \sum_{i=1}^n \biggl(\int_0^\tau Z_i(t) \frac{h(X_i(t))}{\sigma^2(t)} dY_i(t) - \frac{1}{2}
   \int_0^\tau Z_i(t)\frac{h^2(X_i(t))}{\sigma^2(t)} dt \biggr).
   \end{split}
 \end{equation*}
The expected (partial) log-likelihood is given 
\begin{align*}
  \Lambda(h) &= E\biggl(\int_0^\tau Z(t) \frac{h(X(t))}{\sigma^2(t)} dY(t) - \frac{1}{2}
   \int_0^\tau Z(t)\frac{h^2(X(t))}{\sigma^2(t)} dt \biggr) \\
  &= E\biggl(\int_0^\tau Z(t) \frac{h(X(t))\eta_0(X(t))}{\sigma^2(t)} dt - \frac{1}{2}
   \int_0^\tau Z(t)\frac{h^2(X(t))}{\sigma^2(t)} dt\biggr)
\end{align*}
where the second equality is obtained by taking conditional expectation of $Y(t)$ given $X(t)$, and then taking expectation on $X(t)$.

In this context, we define the empirical inner product and corresponding squared norm as
\[
\langle h_1,h_2 \rangle_n =
E_n \int_0^\tau Z(t)h_1(X(t)) \frac{h_2(X(t))}{\sigma^2(t)}\, dt
\]
and $\|h\|_n^2 = \langle h, h \rangle_n$. The corresponding theoretical quantities are
\[
\langle h_1,h_2 \rangle = E\int_0^\tau
Z(t)h_1({X}(t)) \frac{h_2(X(t))}{\sigma^2(t)}\, dt
\] 
and $\|h\|^2 = \langle h, h\rangle$.

We verify conditions used in the master theorems in the main paper under the following primitive assumptions, which were used in \cite{stone2003statistical} for the same context.

\begin{assumption*}[\textbf{DP}]\label{Diffusion_con}
$(i)$  The function $\eta_0$ is bounded on $\mathcal{U}$.

$(ii)$  There are two positive constants $M_2 \geq M_1$ such that $M_1 \leq
  \sigma^{-2}(t) \leq M_2$ whenever $Z(t) = 1$. 

$(iii)$ There are constants $M_4 \geq M_3 > 0$ such that 
\begin{equation*}
    M_3 |A| \leq E\biggl(\int_0^\tau Z(t)\, \mathrm{I}(X(t) \in A)\biggr) \leq M_4 |A| 
\end{equation*}
for all Borel subset $A$ of $\mathcal{X}$, where $|A|$ denotes the Lebesgue measure of $A$.
\end{assumption*}

Under these assumptions, the theoretical norm is equivalent to the $L_2$ norm w.r.t.\ the Lebesgue measure; see (3.1) of \cite{stone2003statistical}. Moreover, a simple calculation implies that \eqref{eq:2ndder-exp-loglik} holds with $M_1=M_2=1$; see (3.3) of \cite{stone2003statistical}. Condition~$\ref{cond:exp-loglik}$ then follows from Lemma~\ref{lemma:exp-loglik}.

Straightforward calculation gives
\begin{equation}\label{1st-der_1_diffus}
\dot{l}[\bar\eta_n; h](\mbf{W}_1) = \int_0^\tau Z_1(t) \frac{h(X_1(t))}{\sigma^2(t)} \{dY_1(t) - \bar{\eta}_n(X_1(t)) \, dt\}.
\end{equation}
Using
\[
dY_1(t) - \bar{\eta}_n(X_1(t)) \,dt = d W_1(t) + \{\eta_0(X_1(t)) - \bar{\eta}_n(X_1(t))\} \,dt,
\]
we obtain that
\begin{align*}
& \mathrm{Var} (\dot{l}[\bar\eta_n; h](\mbf{W}_1)) \\
& \qquad \leq 2 E\biggl(\biggl[\int_0^\tau Z(t) \frac{h(X(t))}{\sigma^2(t)}\, dW_1(t) \biggr]^2\biggr) \\
& \qquad \qquad + 2 E\biggl(\biggl[\int_0^\tau Z(t) \frac{h(X(t))}{\sigma^2(t)}[\eta_0(X(t)) -
  \bar{\eta}_n(X(t))] dt \biggr]^2\biggr).
\end{align*}
Note that $E [\{\int_0^\tau f(t, X(t) \, d W_1(t) \}^2] = E\{\int_0^\tau f^2(t, X(t))\,dt \}$. 
Because of the boundedness of $\eta_0$ and $\bar\eta_n$ and Assumption ({\bf DP})(ii), the right side of the above inequality is bounded above by
\[
2 E\biggl(\int_0^\tau Z(t) \frac{h^2(X(t))}{\sigma^4(t)} dt\biggr)
+ M_1 E\biggl(\int_0^\tau Z(t) \frac{h^2(X(t))}{\sigma^4(t)} dt\biggr) 
\leq M_2 \|h\|^2.
\]
Therefore, the condition in our Lemma~\ref{lemma:score} holds and thus Condition~$\ref{l_12der}(i)$ follows from the lemma.

Finally, (3.4) of \cite{stone2003statistical} indicates that Part (ii) of the condition in Lemma~\ref{lemma:hessian} holds, and thus Condition~\ref{l_12der}($ii$) holds according to the lemma.

\section{Application VII: spectral density estimation for a stationary time series}\label{sec:sden}

We follow the logspline spectral density estimation formulation presented in \citet{kooperberg1995rate}. Consider a stationary linear time series $\{X_t\}$ taking the form
\begin{equation*}
  X_t = \sum_{j = -\infty}^\infty a_j Z_{t-j},
\end{equation*}
where $\{Z_j\}_{j=-\infty}^\infty$ is an independent Gaussian white noise sequence
with mean zero and variance $\sigma^2$.
The theoretical spectral density function $f(\cdot)$ for $\{X_t\}$ is given by
\begin{equation*}
  f(\lambda) = \frac{\sigma^2}{2\pi} \biggl| \sum_{j=-\infty}^\infty
  a_j \exp(-ij\lambda)\biggr|^2, \quad -\pi \leq \lambda \leq \pi.
\end{equation*}
(It is actually well-defined on $\mathbb{R}$ as a periodic function with period $2\pi$.)
The spectral density function is always positive. To ensure the
positivity of its estimator, we consider estimating directly the log
spectral density function $\eta(\lambda) = \log
f(\lambda)$ and then apply the back-transformation $f(\lambda) = \exp \eta(\lambda)$.

Let $X_0, \ldots, X_{T-1}$ be a realization of length $T$ of the
time series. The periodogram is defined as
\begin{equation*}
  I^{(T)}(\lambda) = (2\pi T)^{-1} \biggl|\sum_{t=0}^{T-1} \exp(-i\lambda
  t)X_t\biggr|^2, \quad -\pi \leq \lambda \leq \pi
\end{equation*}
Write
\begin{equation}
 \label{dist_periodogram}
  I^{(T)}(\lambda_k) = f(\lambda_k)W_k, \quad \lambda_k = \frac{2\pi
    k}{T}, \quad \quad k = 0, \ldots, [T/2],
\end{equation}
where $W_k$, $k = 0, \ldots, [T/2]$, are the ratios of the periodogram and the spectral density
function evaluated at the grid points $\lambda_k$ between $[0,\pi]$.
According a standard result in time series analysis
\citep{series1991theory}, the asymptotic distribution of $W_k$ is free
of $f(\cdot)$, which is the exponential distribution with mean one when $\lambda_k$ is not on the boundary of $[0,\pi]$, $W_0$ and $W_{[T/2]}$ (if $T$ is even) have approximately the $\chi^2$ distribution with degree
of freedom one, and $W_0, W_1, \ldots, W_{[T/2]}$ are asymptotically
independent.

Since the spectral density function is symmetric about zero on $[-\pi,\pi]$ and is
periodic (with the period $2 \pi$), it is sufficient to model its segment on $[0,\pi]$ with additional constraints that $f'(0) = f'''(0) = f'(\pi)
= f'''(\pi) = 0$ and $\eta'(0) = \eta'''(0) = \eta'(\pi)
= \eta'''(\pi) = 0$. Letting $\bbG_1$ be a space of splines without constraints, we
use the following subspace of $\bbG_1$ as the estimation space
\begin{equation*}
  \bbG = \{g \in \bbG_1 : g'(0) = g'''(0) = g'(\pi) = g'''(\pi) = 0\}.
\end{equation*}

Set $I_k = I^{(T)}(\lambda_k)$,  $k = 1,2, \ldots, [T/2]$. 
According to (\ref{dist_periodogram}) and the asymptotic distribution
of $W_k$'s,  we can write the (approximate)
log-likelihood function of the periodogram for a candidate function $h
\in \bbG$ as 
\begin{equation*}
  \ell(h) = \frac{1}{[T/2]} \sum_{k=0}^{[T/2]}\psi(I_k,\lambda_k,h), 
\end{equation*}
where
\begin{equation*}
  \psi(y,\lambda;h) = \biggl\{\frac{\delta_\pi(\lambda)}{2} -
  1\biggr\}[h(\lambda) + y \exp(-h(\lambda))]
\end{equation*}
for $0 < \lambda \leq \pi$ and $y \geq 0$,  $\delta_\pi(\lambda) =
1$ if $\lambda=\pi$ and $\delta_\pi(\lambda)=0$ otherwise.
Define the (approximate) expected log-likelihood function as
\begin{equation*}
  \Lambda(h) = \frac{1}{[T/2]}
  \sum_{k=0}^{[T/2]} \biggl\{\frac{\delta_\pi(\lambda_k)}{2} -
  1\biggr\}[h(\lambda_k) + E(I_k) \exp(-h(\lambda_k))].
\end{equation*}

The above discussion has cast the spectral density estimation into the
framework of concave extended linear models with 
$\mathbf{W} = (\lambda, I^{[T]}(\lambda))$. It can be seen from
Theorem~8.12 of \cite{schumaker1981spline} that Proposition
\ref{schumaker} holds when the space is changed from $\bbG_1$ to the
periodic spline space $\bbG$.

We next verify conditions used in the master theorems in the main
paper under the following
primitive assumptions, which were used in \cite{kooperberg1995rate} for the same context.

\begin{assumption*}[\textbf{SD}]\label{spectral_density_con}
$(i)$ $\sum_j |a_j|  j^p < \infty$ for some $p > 1/2$.

$(ii)$  The spectral density function $f_0$ is bounded away from zero
and infinity on $[0,\pi]$.  
\end{assumption*}

Under Assumption~\textbf{SD}(i), the spectral density function $f_0$
is $p$-smooth, so is the logarithm $\eta_0 = \log f_0$.  
Assumption~\textbf{SD}(ii) is equivalent to the assumption that $\eta_0$ is bounded. 
 
In this context, we define the empirical inner product and corresponding squared norm as
\[
\langle h_1,h_2 \rangle_n =
\frac{1}{[T/2]} \sum_{k=0}^{[T/2]} I_k h_1 (\lambda_k) h_2 (\lambda_k)
\]
and $\|h\|_n^2 = \langle h, h \rangle_n$. The corresponding theoretical quantities are
\[
\langle h_1,h_2 \rangle = 
\frac{1}{[T/2]} \sum_{k=0}^{[T/2]} E(I_k) h_1 (\lambda_k) h_2 (\lambda_k)
\]
and $\|h\|^2 = \langle h, h\rangle$. 

By Theorem 10.3.1 of \cite{series1991theory}, $E(I_k) = f(\lambda_k) + O(T^{-1})$, where
$O(T^{-1})$ is uniform in $\lambda_k$. Assumption~\textbf{SD}(ii)
implies that $f(\lambda_k)$ is bounded uniform in $\lambda_k$.
Therefore, the theoretical norm $\|h\| $ is equivalent to $\|h\|_2$,
the $L_2$-norm w.r.t.\ the Lebesgue measure; see also (2) of
\cite{kooperberg1995rate}. 
Moreover, it can be shown that the empirical
and theoretical norms are asymptotically equivalent in the sense of  Proposition~\ref{equiv_norm}.

Observe that
\begin{align*}
& \frac{d^2}{d\alpha^2} \Lambda(h_1+\alpha h_2)\\
& \qquad = \frac{1}{[T/2]} \sum_{k=0}^{[T/2]}
\biggl\{\frac{\delta_\pi(\lambda_k)}{2} -  1\biggr\}
E(I_k) h_2^2(\lambda_k) \exp[-h_1(\lambda_k) -\alpha h_2(\lambda_k )].
\end{align*}
If $\|h_1\|_\infty \leq C$, $\|h_2\|_\infty \leq B$, then
\[
\exp(-C-B) \leq \exp[-h_1(\lambda_k) -\alpha h_2(\lambda_k )]\leq
\exp (C+B).
\] 
It follows that \eqref{eq:2ndder-exp-loglik} holds. Condition~$\ref{cond:exp-loglik}$ then follows from Lemma~\ref{lemma:exp-loglik}.

By definition of the likelihood,
\begin{equation*}
  \label{1st-der_1_sden}
\dot{l}[\bar\eta_T;h](\mathbf{W}_k) = 
  \{\delta_\pi(\lambda_k)/2 - 1\}h(\lambda_k) I_k
  \exp\{-\bar\eta_T(\lambda_k)\}.
\end{equation*}
Notice that $\dot{l}[\bar\eta_T;h](\mathbf{W}_k)$ are independent but
not identically distributed. We need the following weaker version of
Lemma~$\ref{lemma:score}$, which can be proved by a slight modification of
the proof of Lemma~$\ref{lemma:score}$.

\begin{lemma}~\label{lemma:score2}
If there exists a constant $M$ such that for any $h$ satisfying
$\|h\|^2 =1$, it holds that 
\[
\frac{1}{[T/2]} \sum_{k=0}^{[T/2]} \mathrm{Var}\{\dot{l}[\bar\eta_T;h](\mathbf{W}_k) \leq M,
\]
then Condition~$\ref{l_12der}$ $(i)$ holds.  
\end{lemma}

Assume that $\bar\eta_T$ is bounded. We have that, for $h$ with
$\|h\|^2 =1$,
\begin{align*}
 &\frac{1}{[T/2]}\sum_{k=0}^{[T/2]} \mbox{Var}\{\dot{l}[\bar\eta_T;h](\mathbf{W}_k)\} \\
 &\qquad = \frac{1}{[T/2]}\sum_{k=0}^{[T/2]} \{\delta_\pi(\lambda_k)/2 - 1\}^2h^2(\lambda_k) \mbox{Var}(I_k) \exp\{-2\bar\eta_T(\lambda_k)\} \\
  &\qquad \leq M \frac{1}{[T/2]}\sum_{k=0}^{[T/2]} h^2(\lambda_k) \leq M \|h\|^2 = M;
\end{align*}
the first inequality makes use of the boundedness of
$\mbox{Var}(I_k)$, which is guaranteed by Theorem 10.3.2 (ii) of
\citet{series1991theory}, the second inequality follows from the equivalence of the theoretical
norm and the $L_2$ norm with respect to the Lebesgue measure. Consequently, Condition~$\ref{l_12der}$ $(i)$
holds by applying Lemma~\ref{lemma:score2}.

If $\bar\eta_T$ is bounded and $g$ is bounded, 
\begin{align*}
& \frac{d^2}{d\alpha^2}\ell(\bar\eta_T+ \alpha g) \\
&\qquad = \frac{1}{[T/2]}
  \sum_{k=0}^{[T/2]} \biggl\{\frac{\delta_\pi(\lambda_k)}{2} -
                                                           1\biggr\} I_k g^2(\lambda_k)\exp\{-\bar\eta_T(\lambda_k) - \alpha g(\lambda_k)\}\\
&\qquad \lesssim - \frac{1}{[T/2]}  \sum_{k=0}^{[T/2]} I_k g^2(\lambda_k) 
= -  \|g\|_n^2,
\end{align*}
This together with the asymptotic equivalence of the empirical and
theoretical norms implies that Part (ii) of the condition in Lemma~\ref{lemma:hessian} holds, and thus Condition~$\ref{l_12der}$($ii$) holds according to the lemma.

\section{Multivariate case: penalized tensor product splines}~\label{sec:ts}
In this section, we develop results for estimating multivariate
functions using penalized tensor product splines. The development is
in parallel with that for estimating univariate functions using
splines presented in the main paper. 

To obtain our results, it is critical to extend the propositions in
Section~2 of the main paper to the multivariate case. We will make
heavy use of existing mathematical results in the
monograph~\cite{schumaker1981spline}, which will be referred to as
\textbf{S1981} for short for the rest of this section. Note there is a
slight difference in our notations: We use $m$ to denote the degree of
splines for consistency with our main paper, while \textbf{S1981} used
$m$ to denote the order of splines. Our degree-$m$ splines correspond
to order-$(m+1)$ splines in \textbf{S1981}.

\subsection{Summary of results in the univariate case}~\label{sec:univar-spline}
One important step for us to extend the results from the univariate case in the main paper to the multivariate case is to establish the approximation property of tensor product splines, or to extend Proposition~2.1 or Theorem~6.25 of \textbf{S1981} to tensor product splines. We outline in this section the steps in \textbf{S1981} for establishing its Theorem~6.25 and then in the next subsection extend the argument 
to tensor product splines. This subsection also introduces notations that will be used later. 

As in Section~2 of the main paper, we consider a compact interval $[a,b]$ and $k$ interior
knots $t_j, j=1, \dots, k$ in the interval, satisfying $a=t_0<t_1<\dots, t_k
<t_{k+1}=b$. The collection of all degree-$m$ spline functions with these interior knots
forms a linear vector space with dimension $N= m+k+1$. This vector space is referred to as the spline space below. We let $\delta_n = \max_j |t_{j+1}- t_j| $ denote the largest distance between any two neighboring knots, and assume that the knot sequence has the bounded mesh ratio. 

Consider the extended partition $s_1 = \cdots = s_{m+1} = a$,
$s_{m+2}=t_1, \ldots$, $s_{m+k+1}=t_k$, $s_{m+k+2} = \cdots =s_{2m+k+2}=b$.
Let $\{\tau_{ij} = s_i + (s_{i+m+1} - s_i) \frac{(j-1)}{m}: j =
1,\ldots, m+1\}$ be points equally located between $s_{i}$ and $s_{i+m+1}$.
For any bounded function $f$, let $[\tau_{i1},\ldots, \tau_{ij}] f$ be
its $(j-1)$th order divided difference over the points $\tau_{i1}, \dots, \tau_{ij}$ (Definition 2.49, \textbf{S1981}). Then define the dual functionals as 
\begin{eqnarray}~\label{eq:dual-functional}
  \label{dual_fd}
  \lambda_i f = \sum_{j=1}^{m+1} \alpha_{ij}\, [\tau_{i1},\ldots,
  \tau_{ij}] f,
\end{eqnarray}
where the coefficient $\alpha_{ij}$ depends on $s_{i+1},\ldots, s_{i+m}$ and
$\tau_{i1},\ldots,\tau_{i (j-1)}$ (Eqs. 6.38--6.39, \textbf{S1981}).
Let $N_i^{[m]}(\cdot)$ be the normalized B-splines of degree $m$ associated with the knots $s_i, \dots, s_{i+m+1}$ (Definition 4.19, \textbf{S1981}). The dual functionals $\lambda_i$ satisfy
$\lambda_i N_j^{[m]} = \sbf{1}(i=j)$ (see the discussion following Eq. 4.90, \textbf{S1981}).

Next, define a linear operator $Q$,
\begin{equation}\label{eq:interpolant}
Qf(x) = \sum_{i = 1}^{m+k+1} (\lambda_i f) N_i^{[m]}(x).
\end{equation}
The operation $Q$ satisfies two properties: 1. It maps a bounded function $f$ to the spline space; 2. It is invariant to any polynomial of degree $m$, that is, for any polynomial function $f$ of degree $m$, $Qf = f$ (Theorem 6.18, \textbf{S1981}). 

Finally, the function $\eta^*= Q\eta_0$ can be shown to have the
approximation properties given in Proposition~2.1 of the main
paper. As in \textbf{S1981}, let $D^r$ denote the differential operator
so that $D^rf(x)$ is the $r$th derivative of the function $f(\cdot)$ at $x$.
Let $L^p_\infty[a, b]$ denote the Sobolev space
$\{f: D^pf \in L_\infty[a,b]\}$. Assume $1 \leq p \leq m+1$. There is a constant $C$ such that, for all $f$ in $L^p_\infty[a, b]$, there exists a polynomial $p_f$ of degree $m$ such that $\|D^j(f - p_f)\|_\infty \leq C (b-a)^{p-j}$ for all $j \leq p-1$ (Theorem 3.20, \textbf{S1981}). Based on this and the triangular inequality, the upper bounds for the quantities $D^\alpha(\eta_0 - Q\eta_0)$ and $D^\alpha Q\eta_0$ can be studied through $D^\alpha Q(\eta_0-p_{\eta_0})$ on any local interval between two extended knots (Theorem 6.24, \textbf{S1981}) and then these bounds are extended to the whole interval (Theorem 6.25, \textbf{S1981}). 

We end this subsection by presenting and proving some properties of
the dual functionals defined in \eqref{eq:dual-functional}. To
simplify notation, we drop the subscript of $\delta_n$ in the next lemma.

\begin{lemma}\label{dual}
Consider the linear functionals $\lambda_i$ defined in \eqref{dual_fd}.
\begin{enumerate}
\item For all $1 \leq
i \leq m + k + 1$ and $1 \leq j \leq m + 1$, 
\[
\lvert \alpha_{ij} \rvert \leq (m + 1)^{j-1} \delta^{j-1}. 
\]      
\item For all $f \in C^{p-1}[a, b]$,
\[  
\lambda_i f = \sum_{j=1}^p \alpha_{ij} \frac{\partial^{j-1}}{\partial
  x} f(\theta^{ij}) + \sum_{j = p+1}^{m+1} \widetilde{C}_j \alpha_{ij} \delta^{p-j}
\sum_{\nu = 0}^{j - p} \tilde{c}_{j \nu} \frac{\partial^{p-1}}{\partial
  x}f(\xi^{i \nu}),
\]
for some $\theta^{ij} \in [s_i, s_i + (s_{i+m+1}-s_i) \frac{(j-1)}{m}]$,
$\xi^{i \nu} \in [s_i+(s_{i+m+1}-s_i)\frac{\nu}{m},
s_i+(s_{i+m+1}-s_i)\frac{(\nu + p - 1)}{m}]$, where $\widetilde{C}_j$ is a constant that depends on spline degree
$m$ and satisfies $\sup_j \widetilde{C}_j < \infty$, and $\tilde{c}_{j \nu}$ is
a constant that depends on $p$ and satisfies $\sup_{j, \nu} \tilde{c}_{j\nu} < \infty$.
\end{enumerate}
\end{lemma}

\begin{proof}[Proof of Lemma~\ref{dual}]
The first statement has been established in Lemma 6.19 of \textbf{S1981}.
Here we need only to prove the second statement.
The desired expression can be easily shown by using two basic results about divided differences. 

Result 1.  If $f \in C^{r}[a,b]$, $r\geq 1$, then for any points $a = \tau_1 \leq \ldots \leq \tau_{r+1} = b$, 
\begin{eqnarray}
\label{fd_1}
[\tau_1, \ldots, \tau_{r+1}] f = \frac{D^rf(\theta)}{r!}, \mbox{ for some } a
\leq \theta \leq b,
\end{eqnarray}
where $ D^r f$ denotes the $r$th derivative of $f$.

Result 2. Let $r \geq 1$ and $\tau_1, \ldots, \tau_{r+1}$ be  any
equally spaced points. Define $\gamma_j = \tau_{j+1} -
\tau_1$. Because of equal spacing, we have $\tau_{\nu + j +1} -
\tau_{\nu+1} = \gamma_j$ for $ 0 \leq \nu \leq r-1 $, $1\leq j \leq r - \nu$. 
 For any sufficiently smooth function $f$, and for $i$ satisfying $0 \leq i \leq r - 1$,
\begin{eqnarray}
\label{fd_2}
[\tau_1, \ldots, \tau_{r+1}] f = \sum_{\nu=0}^{r-i}\frac{(-1)^{r - i +\nu} {{r -
      i}\choose \nu}[\tau_{\nu+1},\ldots, \tau_{\nu + 1 +i}]f}{\gamma_{i+1}\cdots \gamma_r}.
\end{eqnarray}

Result 1 and Result 2 for $i=r-1$ follow from Theorem 2.51, \textbf{S1981}.
In particular, when $i = r - 1$, \eqref{fd_2} reduces to
\begin{eqnarray}\label{fd_3}
  [\tau_1, \ldots, \tau_{r+1}] f = \frac{[\tau_2, \ldots, \tau_{r+1}]f - [\tau_1, \ldots, \tau_{r}]f}{\gamma_r},
\end{eqnarray}
which is Eq.\ (2.91) of the cited theorem. 

We use mathematical induction to show that \eqref{fd_2} holds for general cases of $i$ and $r$. 
Assume \eqref{fd_2} holds for $i = k$ and all $r$ satisfying $k
\leq r - 1$, so that
\begin{eqnarray}
\label{fd_4}
[\tau_1, \ldots, \tau_{r+1}] f = \sum_{\nu=0}^{r- k}\frac{(-1)^{r - k +\nu} {{r -
      k}\choose \nu}[\tau_{\nu+1},\ldots, \tau_{\nu + k + 1}]f}{\gamma_{k+1}\cdots \gamma_r}.
\end{eqnarray} 
Applying\eqref{fd_3} with $\{\tau_1, \dots, \tau_{r+1}\}$  replaced by
$\{\tau_{\nu + 1}, \ldots, \tau_{\nu + k + 1}\}$ to obtain
\[
[\tau_{\nu + 1}, \ldots, \tau_{\nu + k + 1}] f = \frac{[\tau_{\nu + 2}, \ldots,
\tau_{\nu+k+1}]f - [\tau_{\nu+1}, \ldots,
\tau_{\nu+k}]f}{\gamma_k}.
\] 
Plugging this into the right side of \eqref{fd_4} and
recalculating the coefficients for the common term by making use of a basic property of
combination number ${n  \choose m} = {{n-1} \choose m - 1} + {{n-1}
  \choose m}$, we obtain the equality 
\[
[\tau_1, \ldots, \tau_{r+1}] f =
\sum_{\nu=0}^{r- k + 1}\frac{(-1)^{r - k + 1 +\nu} {{r -
      k + 1}\choose \nu}[\tau_{\nu+1},\ldots, \tau_{\nu + k}]f}{\gamma_{k}\cdots \gamma_r}.
\] 
Thus, \eqref{fd_2} holds for $i = k - 1$ and all $r$ satisfying $k \leq r - 1$. This completes the mathematical induction.

Now we are ready to prove the second statement of the lemma. Recall the expression of $\lambda_i f$ in \eqref{dual_fd}. For $f \in C^{p-1}[a, b]$, applying \eqref{fd_1} to the first $p$ summands we get the first summation term on the right hand side of the equation. For the remaining terms corresponding to $j = p + 1, \ldots, m + 1$, there is no direct link between the divided difference and the derivatives of the function. We first make use of \eqref{fd_2} to reduce the order of the divided differences from $j$ to $p$ by setting $r = j - 1$ and $i = p - 1$, and then apply \eqref{fd_1} to these divided differences each with order $p$. The second summation term on the right hand side of the equation then follows.
\end{proof} 


\subsection{Approximation property of tensor product splines}~\label{sec:approx-ts}

We consider tensor product splines defined on $d$ dimensional
hyper-rectangular domain $\Omega$, which for simplicity is assumed to
be $\Omega = [a, b]^d$. Extension to $\Omega = [a_1, b_1] \times \dots
\times [a_d, b_d]$ is straightforward with some notational
complications.
To give a concrete definition, we use the same knots for each dimension, which are assumed to be $a = t_0 < t_1 < \ldots < t_k < t_{k+1} = b$. For these knots, let 
$N_i^{[m]}(x_j), i=1, \dots, m+k+1,$ be the normalized degree-$m$ B-spline basis functions for dimension $j$ (or variable $j$). The linear space of degree-$m$ tensor product splines on $\Omega$,  
is spanned by the basis functions
\[
 N_{i_1 \cdots i_d}(x_1,\ldots, x_d) = N_{i_1}^{[m]}(x_1) \cdots N_{i_d}^{[m]}(x_d), \qquad 1\leq i_j\leq m+k+1.
\]
Same as in the previous subsection, we assume that the knots satisfy
the bounded mesh ratio property, and denote $\delta_n = \max_j
|t_{j+1} - t_j|$. With some complication of notation, it is
straightforward to extend our results to allow different spline
degrees and different knot placements at different dimensions (or for
different variables $x_j$'s). Let  $\mathcal{G}_n$ be the tensor
product spline space defined above, whose dimension depends on the
size of $\delta_n$ and thus is allowed to grow with the sample size.

We extend the linear operator defined in \eqref{eq:interpolant} to the multivariate case.
To this end, we first define the dual functionals of the tensor product spline space
through composition of dual functionals of the univariate spline spaces along each dimension,
i.e., 
\[
\lambda_{i_1 \cdots i_d} f = \lambda_{i_1}^{(1)} \circ
\lambda_{i_2}^{(2)} \circ \cdots \lambda_{i_d}^{(d)} f, \qquad 1\leq i_j \leq m+k +1,
\]
where $\{\lambda_i^{(j)}\}_{i=1}^{m+k+1}$ are the dual functionals along
dimension $j$ as defined in the previous subsection.
For bounded function $f$ on $\Omega$, let 
\begin{equation}
\label{interpolation}
Qf = \sum_{i_1 = 1}^{m+k+1} \cdots \sum_{i_d = 1}^{m+k+1}
(\lambda_{i_1 \cdots i_d} f) N_{i_1 \cdots i_d}(x_1, \ldots, x_d).
\end{equation}
This is a linear operator that maps bounded functions to tensor product splines.

In our estimation problem, the true unknown function $\eta_0$ is assumed to be in a $L_2$ space with
``smoothness" $p$, which is essentially a classical Sobolev space of
multi-dimensional functions (see
Eq. 13.14 and Eq. 13.25 of \textbf{S1981}),
\begin{eqnarray} \label{sobolev}
  L_\infty^p(\Omega) = \bigg\{f: \sum_{l=0}^p \sup_{|\sbf{\alpha}|=l} \|
  D^{\sbf{\alpha}} f \|_{\infty, \Omega}   < \infty \bigg\},
\end{eqnarray}
where $\sbf{\alpha} = (\alpha_1, \ldots, \alpha_d)$ in a multi-index,
$D^{\sbf{\alpha}} = D^{\alpha_1} \cdots D^{\alpha_d}$, and $|\sbf{\alpha}| =
\alpha_1 + \cdots + \alpha_d$. For any $f \in L_\infty^p(\Omega)$, define $\|f\|_{p, \infty} = \sum_{l=0}^p \sup_{|\sbf{\alpha}|=l} \|
  D^{\sbf{\alpha}} f \|_{\infty, \Omega}$. 
In the one-dimensional case, $d=1$, the Sobolev space in
\eqref{sobolev} reduces $L^p_\infty[a, b] = \{f: D^pf \in L_\infty[a,b]\}$,
which is a subset of the Sobolev space $W^p$ used in the main
paper. Assuming that the unknown function is in $L_\infty^p[a,b]$ instead
of $W^p[a,b]$ imposes a stronger smoothness assumption. This is the price we paid
for extending results to multi-dimension.

The following result gives the approximation property of tensor product splines. It extends Theorem 6.25 of \textbf{S1981}. 
\begin{thm}
\label{main}
Assume $m \geq p-1$.  If $\eta_0 \in L_\infty^p(\Omega)$, then
  \begin{equation*}
  \begin{rcases}
    \underset{|\sbf{\alpha}| = 0, \ldots, p - 1}{\|D^{\sbf{\alpha}}(\eta_0 - Q\eta_0)\|_\infty} \\
    \underset{|\sbf{\alpha}| \geq p, \alpha_i \leq m}{\|D^{\sbf{\alpha}} Q\eta_0\|_\infty}
  \end{rcases}
  \leq C \delta^{p - |\sbf{\alpha}|}_n \|\eta_0\|_{p, \infty}.
\end{equation*}
The constant $C$ depends on $d$, $p$, $m$ and $\eta_0$.
\end{thm}

The proof of Theorem \ref{main}, to be given at the end of this subsection, follows the same steps used for
establishing the result for the univariate case, as outlined in subsection~\ref{sec:univar-spline}.

The following proposition is a multi-dimensional analog of
Proposition~2.1 of the main paper. It is an immediate corollary of Theorem~\ref{main}. A multi-dimensional version of the
penalty functional is defined as
\begin{equation}\label{eq:penalty-multivar}
J_q(f) = \sum_{|\sbf{\alpha}| = q} C_{\sbf{\alpha}}
J_{\sbf{\alpha}}(f),
\end{equation}
where $C_{\sbf{\alpha}} = q!/(\alpha_1! \dots \alpha_d!)$, $J_{\sbf{\alpha}}(f) = \int_\Omega |D^{\sbf{\alpha}}f(x_1,\ldots, x_d)|^2\, dx_1 \cdots dx_d$.

\begin{prop}\label{prop:approx}
Assume $\eta_0 \in L_\infty^p(\Omega)$ and $m \geq {p}-1$.  
There exist a function $\eta^*_n \in \mathcal{G}_n$ and
constants $C_1$--$C_3$,  depending on $d$, $p$, $m$ and $\eta_0$ such that
\begin{equation*}
 \|\eta^*_n - \eta_0\|_2 \leq \|\eta^*_n - \eta_0 \|_\infty \leq C_2 \delta_n^{p},
\end{equation*}
and moreover, if $q \leq m$, then  $J_q(\eta^*_n) \leq C_3 \delta_n^{2(p - q) \wedge 0}.$
\end{prop}

The proof of Theorem~\ref{main} will make use of a result about the approximation
property of polynomials. To present this result, we first give a
definition of Taylor expansion for a multi-dimension function.
Let $\Delta$ be a $d$-dimensional hyper-rectangle whose longest side
has length $\eta$. Let $\sbf{B}\subset \Delta$ be an open ball. Let $\psi\in
C_0^\infty(\sbf{B}) $ be a test function with $\int_{\sbf{B}} \psi
=1$. For a fixed integer $r$, define the total Taylor expansion of $f$ with respect to  $\psi$ as
\begin{equation}\label{eq:taylor}
  T_\psi f(\sbf{x}) = \sum_{|\sbf{\alpha}| \leq r} \int_{\sbf{B}}(-1)^{|\sbf{\alpha}|}D^{\sbf{\alpha}}\bigg[\psi(\sbf{y})\frac{(\sbf{x} - \sbf{y})^{\sbf{\alpha}}}{\sbf{\alpha}!}\bigg]f(\sbf{y})\, d\sbf{y},
\end{equation}
which is a polynomial of total degree $r$  (Eq. 13.33, \textbf{S1981};
we converted total order to total degree). 
The following lemma is a special case of Theorem 13.20 of \textbf{S1981} by taking $\Lambda = \{\boldsymbol{\alpha}: |\boldsymbol{\alpha}| < p\}$.

\begin{lemma}
  \label{approx_polynomial}
There exists a constant $C$ (depending only on $d$, $p$ and $\psi$) such that for all $f \in L_\infty^p(\Delta)$ and $|\sbf{\beta}| < p$, 
\begin{equation*}
\| D^{\sbf{\beta}}(f - T_{\psi}f) \|_\infty \leq C \eta^{p - \mid \sbf{\beta} \mid} \sum_{\sbf{\alpha} : \mid \sbf{\alpha} + \sbf{\beta}  \mid = p} \| D^{\sbf{\alpha} + \sbf{\beta}} f\|_\infty,
\end{equation*}
where $T_{\psi} f$ is the polynomial of total degree $r = p - 1$ defined by
the total Taylor expansion \eqref{eq:taylor}. 
\end{lemma}

\begin{proof}[Proof of Theorem \ref{main}]
Following the same strategy as used in the univariate case and summarized
in Section~\ref{sec:univar-spline}, we first obtain a bound on each hyper-rectangle $[t_{1_j},
t_{1_{j+1}}] \times \cdots \times [t_{d_j}, t_{d_{j+1}}]$ for $i_j =
0, \ldots, k$ with $t_{i_0} = a, t_{i_{k+1}} = b$, $i =1, \dots, d$, and then combine the
results. 

Because of the compositional definition of $Qg$, it is easy to see that $Qg = g$ whenever $g$ is a polynomial of total degree $m$. In fact, we can treat one dimension at a time with the rest dimensions fixed and use the invariant property of the operator at each dimension as presented after equation \eqref{eq:interpolant}. 


We have the following two identities
  \begin{align}
    D^{\sbf{\alpha}}(\eta_0 - Q\eta_0) &= D^{\sbf{\alpha}}(\eta_0 - T_\psi \eta_0 ) 
    - D^{\sbf{\alpha}} Q (\eta_0 - T_\psi \eta_0 ),   \mbox{ if }
                                         |\sbf{\alpha}| < p, \label{link1} \\
    D^{\sbf{\alpha}}Q\eta_0 &=  D^{\sbf{\alpha}} QT_\psi \eta_0  +
                              D^{\sbf{\alpha}} Q (\eta_0 - T_\psi
                              \eta_0 ),   \mbox{ if } 
                              |\sbf{\alpha}| \geq p. \label{link2} 
  \end{align}
We can apply Lemma~\ref{approx_polynomial} to bound the
$L_\infty$-norm of the first term on the right hand side of
(\ref{link1}). On the other hand, since $T_\psi\eta_0$ is a polynomial of total degree
$p-1$ ($\leq m$), we have that $QT_\psi\eta_0 = T_\psi \eta_0$ and
$D^{\sbf{\alpha}}QT_\psi \eta_0 = 0$ for any $|\sbf{\alpha}| \geq p$,
thus the first term on the right hand side of (\ref{link2}) becomes
zero. Therefore, it remains to obtain an upper bound of $D^{\sbf{\alpha}} Q(\eta_0 -
T_\psi \eta_0)$ on each hyper-rectangle $[t_{1_j}, t_{1_{j+1}}] \times \cdots \times
[t_{d_j}, t_{d_{j+1}}]$. According to the definition of $Q$ given in (\ref{interpolation}), we need
only to bound $\lambda_{i_1 \cdots i_d}(\eta_0 - T_\psi\eta_0)$ and $D^{\sbf{\alpha}}
\{\prod_{j=1}^d N_{i_j}^{[m]}(x_j)\}$. This can be achieved by
repeatedly applying Lemma~\ref{dual}, Lemma~\ref{approx_polynomial} and
the triangle inequality, as shown below. 

For simplicity, we present only details of the proof for $d = 2$. The general case
follows the same argument but with more complicated notations. 
Denote $\eta_{0c} = \eta_0 - T_\psi\eta_0$. 
Applying Lemma \ref{dual} twice on
$\lambda_{i_1 i_2}\eta_{0c} = \lambda_{i_1}^{(1)} \circ
\lambda_{i_2}^{(2)} \eta_{0c}$, one on each dimension, we obtain
\begin{equation}  \label{dual_dim2}
\begin{split}
& \lambda_{i_1}^{(1)} \circ \lambda_{i_2}^{(2)} \eta_{0c}  \\
= & \sum_{j_2
= 1}^p \alpha_{i_2 j_2}^{(2)} \sum_{j_1 = 1}^{p - j_2 + 1} \alpha_{i_1
j_1}^{(1)} \frac{\partial^{j_1 - 1}}{\partial x_1}
\frac{\partial^{j_2-1}}{\partial x_2} \eta_{0c}(\theta_1^{i_1 j_1},
\theta_2^{i_2 j_2}) \\ 
& + \sum_{j_2 = 1}^p \alpha_{i_2
j_2}^{(2)} \sum_{j_1 = p - j_2 + 2}^{m_1 + 1} \widetilde{C}_{j_1}
\alpha_{i_1 j_1}^{(1)} \delta_n^{p- j_2 -j_1 + 1} \\
&\qquad  \sum_{\nu_1=0}^{j_1 -
p + j_2 - 1} \tilde{c}_{j_1 \nu_1}\frac{\partial^{p - j_2}}{\partial x_1}
\frac{\partial^{j_2-1}}{\partial x_2} \eta_{0c}(\xi_1^{i_1 \nu_1},
\theta_2^{i_2 j_2})\\ 
& + \alpha_{i_1 1}^{(1)} \sum_{j_2 =
p+1}^{m+1} \widetilde{C}_{j_2} \alpha_{i_2j_2}^{(2)} \delta_n^{p-j_2}
\sum_{\nu_2 = 0}^{j_2 - p} \tilde{c}_{j_2 \nu_2}
\frac{\partial^{p-1}}{\partial x_2}\eta_{0c}(\theta_1^{i_1 1},
\xi_2^{i_2 \nu_2}).
\end{split}
\end{equation}
Here, we treat the summations above as zero whenever the lower end exceeds
the upper end, e.g. the second summand becomes zero if $m = 0, p = 1$ and so does the third summand if $m = p - 1$. 
From the expansion \eqref{dual_dim2}, the upper bound of
$\lambda_{i_1 i_2} \eta_{0c}$ depends on the upper bound of 
\begin{eqnarray}
  \label{upper_dual_1}
\qquad \quad\sup_{(x_1, x_2) \in \Omega} \bigg\lvert \frac{\partial^{j_1 - 1}}{\partial x_1} \frac{\partial^{j_2-1}}{\partial
x_2} \eta_{0c}(x_1, x_2) \bigg\rvert, \quad \mbox{ if } j_1,
  j_2 \geq 1,  j_1 + j_2 \leq p + 1,
\end{eqnarray}
and that of $\alpha_{i_1 j_1}^{(1)}$, $\alpha_{i_2 j_2}^{(2)}$; the
range for $i_1, j_1, i_2, j_2$ can be found in equations
\eqref{interpolation} and \eqref{dual_dim2}. 
By Lemma \ref{dual}, $\lvert
\alpha_{i_1 j_1}^{(1)} \rvert \leq \{(m + 1) \delta_n\}^{j_1 - 1}$ and $\lvert
\alpha_{i_2 j_2}^{(2)} \rvert \leq \{(m + 1) \delta_n\}^{j_2 - 1}$.
According to Lemma \ref{approx_polynomial}, \eqref{upper_dual_1}
is upper bounded by $C \delta_n^{p - j_1 - j_2 + 2} \|\eta_0\|_{p,\infty}$.
Thus, in view of \eqref{dual_dim2}, these together yield
\begin{equation}
  \label{upper_dual_2}
  |\lambda_{1 i_1} \circ \lambda_{2 i_2} \eta_{0c}| \leq C \delta_n^p \|\eta_0\|_{p,\infty},
\end{equation}
for a constant $C$ depending on $d$, $m$, $p$, $\psi$. 

On the other hand, for any $\sbf{\alpha} = (\alpha_1, \alpha_2)$ with
$\alpha_i \leq m$ for $i = 1,2$, and $x_1 \in (t_{1_j},t_{1_{j+1}}), x_2 \in (t_{2_j},t_{2_{j+1}})$,
\begin{equation}
\label{upper_dual_3}
\begin{split}
\frac{\partial^{\alpha_1}}{\partial
  x_1} \frac{\partial^{\alpha_2}}{\partial x_2} \{N_{i_1}^{[m]}(x_1)
N_{i_2}^{[m]}(x_2)\} & = \frac{d^{\alpha_1}}{d
  x_1} N_{i_1}^{[m]}(x_1) \frac{d^{\alpha_2}}{d x_2} 
N_{i_2}^{[m]}(x_2) \\
& \leq C \delta_n^{-\alpha_1-\alpha_2}.
\end{split}
\end{equation}
The last inequality uses the bound of derivatives of normalized
B-splines given in Theorem 4.22 of \textbf{S1981}.  


Finally, according to (\ref{interpolation}),
  \[
    \frac{\partial^{\alpha_1}}{\partial
  x_1} \frac{\partial^{\alpha_2}}{\partial x_2} Q(\eta_{0c}(x_1,x_2)) = \sum_{i_1,i_2 = 1}^{m+k+1}
(\lambda_{i_1 i_2} \eta_{0c}) \frac{\partial^{\alpha_1}}{\partial
  x_1} \frac{\partial^{\alpha_2}}{\partial x_2}N_{i_1 i_2}(x_1,x_2).
\]
The $L_\infty$ bound of $\frac{\partial^{\alpha_1}}{\partial
  x_1} \frac{\partial^{\alpha_2}}{\partial x_2} Q(\eta_{0} -
T_{\psi}\eta_0)$ on
 $(t_{1_j},t_{1_{j+1}}) \times (t_{2_j},t_{2_{j+1}})$ is established
 by using
the upper bound of $\lambda_{i_1 i_2} \eta_{0c}$ 
given in \eqref{upper_dual_2}, the upper bound of $\frac{\partial^{\alpha_1}}{\partial
  x_1} \frac{\partial^{\alpha_2}}{\partial x_2}N_{i_1 i_2}(x_1,x_2)$
given in \eqref{upper_dual_3}, and the triangle inequality.
\end{proof}

\subsection{Two properties of the tensor product spline space}
We now extend Proposition~2.2 of the main paper to the tensor product
spline space $\mathcal{G}_n$, defined in
Section~\ref{sec:approx-ts}. Extending the univariate spline case, consider 
a measure of the complexity of the tensor product spline space defined as
\begin{equation}
  \label{ts:A_n}
  A_n = \sup_{g \in \mathcal{G}_n, \|g\|_{2} \neq 0}\biggl\{\frac{\|g\|_\infty }{\|g\|_{2}}\biggr\}.
\end{equation}
Proposition~2.2 of the main paper and Lemma 1 of
\cite{huang1998projection} together yield the following result.
\begin{prop}\label{ts_constant}
Under the bounded mesh ratio, $A_n \asymp \delta_n^{-d/2}$.
\end{prop}

Let the empirical and theoretical norms $\|\cdot\|_n$ and $\|\cdot\|$ be defined as in Section~2 of
the main paper. We assume that $X$ has a density function which is bounded away from
0 and infinity on $\Omega$, and consequently the theoretical norm $\|\cdot\|$ is equivalent 
to $\|\cdot\|_2$, the usual $L_2$-norm relative to the Lebesgue measure.
The following proposition has been proved in
\cite{huang2003asymptotics}.

\begin{prop}\label{ts:equiv_norm}
  Under bounded mesh ratio condition, if \\
$\lim_n N_n \log(n) /n = 0$, then the
  empirical and theoretical norms are asymptotically equivalent, that
  is,
 \[
\sup_{g \in \mathcal{G}_n, \|g\|\not = 0} \bigg| \frac{\|g\|_n}{\|g\|} - 1\biggr| = o_P(1).
\]
\end{prop}

\subsection{The penalty functional}
Recall the definition of the penalty functional defined in \eqref{eq:penalty-multivar}
\begin{equation*}
J_q(f) = \sum_{|\sbf{\alpha}| = q} C_{\sbf{\alpha}}
J_{\sbf{\alpha}}(f),
\end{equation*}
where $C_{\sbf{\alpha}} = q!/(\alpha_1! \dots \alpha_d!)$,
$J_{\sbf{\alpha}}(f) = \int_\Omega |D^{\sbf{\alpha}}f(x_1,\ldots,
x_d)|^2\, dx_1 \cdots dx_d$. Define the quadratic functional
\[
V(\phi_\nu,\phi_\mu) = \int_{\Omega} \phi_\nu(x) \phi_\mu(x)
\omega(x) \, dx,
\]
where $\omega(\cdot)$ is a non-negative weight function.

Applying Theorem 3.1 of \cite{weinberger1974variational}, it can be
shown that, if $V$ is
completely continuous with respect to $J$, then $V$ and $J$ can be
simultaneously diagonalized in the following sense 
\citep[see Section 9.1 of][]{gu2013smoothing}. There exists a sequence of eigenfunctions $\phi_\nu$, $\nu =1, 2, \dots, $ and the associated sequence of eigenvalues
$\rho_\nu\geq 0$ of $J$ with respect to $V$ such that
\[
V(\phi_\nu,\phi_\mu) = \delta_{\nu\mu}, \quad 
J_q(\phi_\nu,\phi_\mu) = \rho_\nu \delta_{\nu\mu},
\]
where $\delta_{\nu\mu}$ is the Kronecker delta, 
\[
V(\phi_\nu,\phi_\mu) = \int_\mathcal{U} \phi_\nu(x) \phi_\mu(x)
\omega(x) \, dx, \quad
J_q(\phi_\nu,\phi_\mu) = \int_\mathcal{U} \phi_\nu^{(q)}(x)
\phi_\mu^{(q)}(x)\, dx.
\]
Furthermore, any function $h$ satisfying $J_q(h) < \infty$
has a Fourier series expansion with the eigen basis
$\{\phi_\nu\}$,
\[
h = \sum_\nu h_\nu\phi_\nu, \quad h_\nu = V(h,\phi_\nu),
\]
and 
\[
V(h) =\sum_\nu h_\nu^2, \quad J_q(h) = \sum_\nu \rho_\nu h_\nu^2.
\]
Therefore,
\[
\|h\|^2 + \lambda_n J_q(h) = (V + \lambda_n J)(h) 
= \sum_\nu (1 + \lambda_n \rho_\nu) h_\nu^2.
\]

The next result is a multi-dimensional version of Proposition~2.4 of
the main paper. It gives the rate of divergence to infinity of the eigenvalues.

\begin{prop}\label{multi:rate_eigenvalue}
Assume $V(h) = \|h\|^2 = \int_{\Omega} h^2(x) \omega(x)\,dx$ for a weight
function $\omega$ that is bounded away from zero and infinity, that is, there exist constants $C_1, C_2 > 0$ such that
  \begin{eqnarray*}
    C_1 \leq \omega(x) \leq C_2, \quad \text{for any} \quad a \leq x \leq b. 
  \end{eqnarray*} 
Then $V$ is completely continuous with respect to $J_q$.
Moreover, we have $0\leq \rho_\nu \uparrow \infty$, and
$\rho_\nu\asymp \nu^{2q/d}$ for all sufficiently large $\nu$.
\end{prop}

The above result follows from Theorem 14.6 of \cite{agmon1965}, as
shown in the proof of Theorem~5.3 of \cite{utreras1988}.

The following result in a multi-dimensional version of
Proposition~2.5. The proof can be found in Lemma~9.1 of \cite{gu2013smoothing}.

\begin{prop}\label{gu-multivar}
Assume there is a constant $C>0$ such that $\rho_\nu \ge C \nu^{2q/d} \,
(q> d/2), $ for all large $\nu$. If $\lambda_n
\rightarrow 0$, as $n\to \infty$, then
\begin{equation}\label{multi:appx-equality}
\sum_{\nu}\frac{1}{1+\lambda_n\rho_{\nu}}=O(\lambda_n^{-1/(2q/d)}).
\end{equation}
\end{prop}

\subsection{Convergence rate of penalized tensor product spline estimators}\label{ConvRate_TP}
In this section, we extend the results in Section~3 of the main paper
to penalized tensor product spline estimators for estimating a
multi-dimensional function. The estimator $\hat\eta_n$ is defined as the
maximizer among $g\in \mathcal{G}_n$ of the penalized likelihood
defined in equation (1) of the main paper, where 
$\mathcal{G}_n$ is the tensor produce spline space defined in
Section~\ref{sec:approx-ts}, and the multi-dimensional penalty functional $J_q(\cdot)$ defined in
\eqref{eq:penalty-multivar} replaces its one-dimensional version given in (2).

The rate of convergence of a penalized tensor product spline estimator depends on
three positive integers:
\begin{itemize}
\item $p$---the smoothness $p$ of the unknown function (i.e., we
  assume $\eta_0\in  L^p_\infty(\Omega)$);
\item $m$---the degree of the tensor product splines in $\mathcal{G}_n$;
\item $q$---the order of the penalty functional $J_q(\cdot)$.
\end{itemize}
To ensure that the penalty functional is well-defined on
$\mathcal{G}_n$, we assume that $q \leq m$. We also assume that 
$q> d/2$ to ensure the eigenvalues of the penalty functional have desired rate of divergence (see Proposition~\ref{gu-multivar}).

Using Propositions~\ref{prop:approx}--\ref{gu-multivar} to replace
Propositions 2.1--2.5 in the main paper, we can extend the results in
Section~3 to the multi-dimensional case in a straightforward way. The
results are summarized below. We need to replace $\mathbb{G}_n$ by
$\mathcal{G}_n$ when converting statements in Section~3 to tensor
product splines.

Theorem 3.1 holds without change in the multi-dimensional case. The
same proof applies, with a slight modification of replacing
Proposition~2.1 by Proposition~\ref{prop:approx}.

Theorem 3.2 needs a slight modification. We modify
Condition~3.2 ($i$) to the following:

\vskip .3em
{\sc Condition} $3.2 (i')$ 
\begin{equation*}
\sup_{g \in \mathcal{G}_n} \;  \frac{|(E_n - E) \; \dot{l}[\bar\eta_n; g]|^2}{\|g\|^2 + \lambda_n
  J_q(g)} = O_P\biggl( \frac{1}{n\lambda_n^{d/(2q)}}
\bigwedge \frac{1}{n\delta_n^d} \biggr).
\end{equation*}
\vskip .3em

By replacing Proposition~2.5 with Proposition~\ref{gu-multivar} in its proof,
Lemma~5.2 can be extended in an obvious manner to provide a sufficient
condition for Condition $3.2(i')$.

A multi-dimensional version of Theorem~3.2 with penalized tensor-product splines
is given below. It can be proved using the same proof of Theorem~3.2. 

\begin{thm}\label{est_err_ts}
  Assume Conditions $3.2 (i')$ and $3.2 (ii)$ hold. If $\lim_n \delta_n \lor \lambda_n =0 $ and 
$$\lim_n
  A_n^2 \allowbreak (\frac{1}{n\lambda^{d/(2q)}_n} \wedge \frac{1}{n\delta_n^d}) = 0,$$
  then $\|\hat\eta_n - \bar\eta_n\|_\infty = o_P(1)$ and 
\[
\|\hat\eta_n - \bar\eta_n\|^2 + \lambda_n J_q(\hat{\eta}_n - \bar{\eta}_n) =
  O_p\biggl(\frac{1}{n  \lambda_n^{d/(2q)}}\bigwedge \frac{1}{n \delta_n^d} \biggr).
\] 
\end{thm}

Combining the results of Theorems~3.1 and \ref{est_err_ts},
we obtain the following result that gives the rate of convergence of
$\|\hat{\eta}_n - \eta_0\|^2$ to zero for the penalized tensor
product spline estimator $\hat\eta_n$.
The result also gives a bound for the size of  $J_{q}(\hat{\eta}_n)$.

\begin{cor}\label{cor:conv_rates_ts}
Assume Conditions~$3.1$, $3.2 (i')$ and $3.2 (ii)$ hold.
 If $\lim_n \delta_n \lor \lambda_n =0 $ and 
\begin{equation}\label{eq:cond_cor_ts}
\lim_n A_n^2 \biggl (\delta^{2{p'}}_n \lor (\lambda_n \delta_n^{2(p'-q) \wedge 0} )
+ \frac{1}{n\lambda_n^{d/(2q)}} \bigwedge \frac{1}{n\delta_n^d}\biggr) = 0,
\end{equation}
 then $\|\hat{\eta}_n - \eta_0\|_\infty = o_P(1)$ and
\begin{equation*}
\|\hat{\eta}_n - \eta_0\|^2 + \lambda_n J_{q}(\hat{\eta}_n) 
= O_p\biggl(\delta_n^{2{p'}} \vee  (\lambda_n \delta_n^{2(p'-q) \wedge 0})   + \frac{1}{n
    \lambda_n^{d/(2q)}}\bigwedge \frac{1}{n \delta_n^d} \biggr).
\end{equation*}
\end{cor}

This result covers all practical combinations of $p$, $q$ and $m$ with
the only restriction being the necessary requirement $q \leq m$ (otherwise the
penalty functional is not defined). Following this result, the
asymptotic behavior of the penalized tensor-product splines can be classified into
seven scenarios as shown in Table~\ref{summary_table_ts}. This table
is an extension of Table~\ref{summary_table} in the main paper to the multi-dimensional case.

\begin{table}[t]
\caption{Seven scenarios for rate of convergence {\rm ($\|\hat\eta_n -
    \eta_0\|^2 + \lambda_n J_q(\hat\eta_n)$)} of penalized tensor
  prodcut spline estimators in dimension $d$}
\label{summary_table_ts}
\begin{center}
\begin{tabular}{c >{\centering\arraybackslash}p{5cm} c}
\hline
\textbf{Rate of convergence} &\textbf{Parameters for achieving the
                                                 best rate} & \textbf{Best rate} \\[1em]
\hline\hline
\multicolumn{3} {c}{I. $ q<p'$ (i.e., $q < p$ and $q < m+1$)}\\[1em]
\hline
\multicolumn{3} {l}{1. $\lambda_n \lesssim \delta_n^{2p'}$} \\[.5em]
 $\delta_n^{2p'} + (n\delta_n^d)^{-1}$
 & $\delta_n\asymp n^{-1/(2p'+d)} $&  $n^{-2p'/(2p'+d)}$  (*) \\
\hline
\multicolumn{3} {l}{2. $\delta_n^{2p'} \lesssim \lambda_n \lesssim
  \delta_n^{2q}$} 
\\[.5em]
 $\lambda_n + (n \delta_n^d)^{-1}$
 & $\lambda_n \asymp \delta_n^{2p'}$,
$\delta_n\asymp n^{-1/(2p'+d)}$ &  $n^{-2p'/(2p'+d)}$  (*)  \\
\hline
\multicolumn{3} {l}{3. $\lambda_n \gtrsim \delta_n^{2q}$}\\[.5em]
$\lambda_n + (n \lambda_n^{d/(2q)})^{-1}$
& $\lambda_n \asymp n^{-2q/(2q+d)}$ &  $n^{-2q/(2q+d)}$ \\
\hline\hline
\multicolumn{3} {c}{II. $ q= p'(=p)$ (i.e., $p = q \leq m$)}\\[1em]
\hline
\multicolumn{3} {l}{1. $\lambda_n \lesssim \delta_n^{2p}$} \\[.5em]
 $\delta_n^{2p} + (n \delta_n^d)^{-1}$ 
& $\delta_n \asymp n^{-1/(2p+d)}$& $n^{-2p/(2p+d)}$ (**) \\
\hline
\multicolumn{3} {l}{2. $\lambda_n \gtrsim \delta_n^{2p}$}\\[.5em]
 $\lambda_n  + (n \lambda_n^{d/(2p)})^{-1}$ 
& $\lambda_n \asymp n^{-2p/(2p+d)}$ & $n^{-2p/(2p+d)}$ (**) \\
\hline\hline
\multicolumn{3} {c}{III. $ q > p' (= p)$ (i.e., $p < q \leq m$)}\\[1em]
\hline 
\multicolumn{3} {l}{1. $\lambda_n \lesssim \delta_n^{2q}$}\\[.5em]
 $\delta_n^{2p} + (n \delta_n^d)^{-1}$ 
& $\delta_n \asymp n^{-1/(2p+d)}$ & $n^{-2p/(2p+d)}$  (**)\\
\hline
\multicolumn{3} {l}{2. $\lambda_n \gtrsim\delta_n^{2q}   $}\\[.5em]
 $ \lambda_n \delta_n^{2p-2q} +  (n \lambda_n^{d/(2q)})^{-1}$ 
&  $\delta_n \asymp \lambda_n^{1/(2q)}$, $\lambda_n \asymp n^{-2q/(2p
  +d)}$& $n^{-2p/(2p+d)}$  (**)\\
\hline
\end{tabular}
\end{center}
(*) achieving Stone's optimal rate when $p'=p$,  (**) achieving Stone's optimal rate 
\end{table}

Using Proposition~\ref{ts_constant}, Condition~\eqref{eq:cond_cor_ts} can be simplified in each
scenario as follows:
\begin{itemize}
\item Cases I.1, II.1, III.1: $p' > d/2$, $n \delta_n^{2d} \to \infty$.
\item Case I.2: $n \delta_n^{2d} \to \infty$, $\lambda_n / \delta_n^d \to
  0$.
\item Cases I.3, II.2, III.2: $n \delta_n^d \lambda_n^{d/(2q)}  \to \infty$  (or its
  sufficient condition $n \delta_n^{2d} \to \infty$), $\lambda_n / \delta_n^d \to
  0$.
\end{itemize}
An overall sufficient condition for all these conditions to hold is 
$p' > 1/2$, $n \delta_n^{2d} \to \infty$, and $\lambda_n / \delta_n^{d+
  2(q - p)\wedge 0} \to 0$.

From Table~\ref{summary_table_ts}, we observe that, similar to the
univariate case, the asymptotic behavior of the penalized
tensor product spline estimators depend on the interplays among the smoothness of
unknown function, spline degree, penalty order, spline knot number,
and penalty parameter.

\section{Penalized bivariate splines on triangulations}~\label{sec:bivariate-splines}
In this section, we develop results for estimating bivariate functions
using penalized bivariate splines defined on triangulations.
Suppose $\Omega$, the domain of the unknown function to be estimated, is a polygonal
domain in $\mathbb{R}^2$. A collection $\Delta$ of triangles forms
a triangulation of $\Omega$, if these triangles form a partition of $\Omega$ and,
if a pair of triangles in $\Delta$ intersect, then their intersection
is either a common vetex or a common edge. A bivariate spline of
degree $m$ on the triangulation $\Delta$ refers to a function which is
a bivariate polynomial of total degree $m$ on each triangle, and the
pieces join together to ensure some degree of global smoothness.
\cite{lai2007spline}, abbreviated below as \textbf{LS2007}, presents a comprehensive mathematical treatment of
polynomial splines on triangulations. 

We consider a sequence of triangulations $\Delta_n$ of $\Omega$, where
$n$ denote the sample size. This family of triangulations is required
to be quasi-uniform, i.e., the ratio of the longest edge and the
inradius of the triangle is bounded above by a universal constant
for all triangles. This quasi-uniform requirement is satisfied if the
smallest angles in the triangulations are bounded away from zero by a positive constant.
(See Remark 4.2, page 122, \textbf{LS2007}.) 

Given $0\leq q \leq m$ and a triangulation $\Delta$, denote the space
of $C^q$ continuous bivariate splines of degree $m$ 
\[
  \mathcal{S}_m^q(\Delta) = \{s \in C^{q}(\Omega): s|_T \in \mathcal{P}_m \mbox{ for all } T \in \Delta\},
\]
where $s|_T$ denotes the restrict of $s$ to the triangle $T$, and
$\mathcal{P}_m$ is the space of bivariate polynomial functions of total degree $m$.
Given sample size $n$, our estimation space $\mathcal{G}_n$ is taken to be $\mathcal{S}_m^q(\Delta_n)$.
Let $\delta_n$ denote the longest edge in the triangulation
$\Delta_n$. 

Definitions of the function space $L_\infty^p(\Omega)$ given in \eqref{sobolev} and
the penalty functional $J_q(f)$ given in \eqref{eq:penalty-multivar}  extend naturally to
polygonal region $\Omega$. 
The formulation of penalized spline estimator also extend naturally to
the current situation with a straightforward substitution of tensor
product spline spaces by bivariate spline spaces. By inspecting the proofs, we conclude that the asymptotic results for penalized tensor product
spline estimators, i.e., Theorem~3.1 of the main paper, 
Theorem~\ref{est_err_ts} and Corollary~\ref{cor:conv_rates_ts}
presented in Section~\ref{ConvRate_TP}, extend to penalized
bivariate spline estimators on triangulations (corresponding to $d=2$), provided that
Propositions~\ref{prop:approx}--\ref{gu-multivar} extend to the
current context.

We now show that Propositions~\ref{prop:approx}--\ref{gu-multivar}
hold for bivariate splines ($d=2$) on triangulations under the additional assumption $m \geq 2q+3$.
This assumption is needed to ensure $\mathcal{S}_m^q$ has a stable
local minimal determine set and thus has optimal approximation
power. If $m< 2q +3$, $\mathcal{S}_m^q$ does not have optimal
approximation power. See page 141 of \textbf{LS2007}. (Here and blow,
we present the cited results using the notation of the current paper.)

Proposition~\ref{prop:approx} adapted to bivariate splines follows from
Theorems 5.18, 5.19 and 10.10 of the monograph \textbf{LS2007}. To
obtain the bound of
$J(\eta^*)$ for $p \leq q \leq m$ 
requires a slight extension of the cited
results, as explained below. Using the notations and equation numbers
from the cited monograph, Theorem 5.19 (and thus Theorem~10.10) there still
holds when $\alpha + \beta >m$ and $D_x^\alpha
D_y^\beta Q f$ replaces $D_x^\alpha D_y^\beta (f-Qf)$ in the
inequality (5.19) of the cited monograph. This
is because, for the degree $m$ polynomial $p$ that satisfies inequality (5.18), 
$Qp = p$ and $D^{\alpha}_{x}D^{\beta}_{y}\, p =0$.
Thus $D_x^\alpha D_y^\beta Q f = D_x^\alpha D_y^\beta Q (f-p)$, and
the system of inequalities in the first paragraph of page 140 of \textbf{LS2007} holds
when $\alpha+ \beta >m$, since the Markov inequality (1.5) in the
cited monograph still applies. 

Proposition~\ref{ts_constant} adapted to bivariate splines follows from the
discussion on page~250 of
\cite{huang1998projection}. Proposition~\ref{ts:equiv_norm} is a
result from \cite{huang2003asymptotics}, which covers bivariate
splines on triangulations.
Proposition~\ref{multi:rate_eigenvalue} is originated from the proof of
Theorem~5.3 of \cite{utreras1988} and holds when the domain $\Omega$
is a polygonal.
Proposition~\ref{gu-multivar} clearly covers $d=2$ as a special case.


\bibliographystyle{imsart-nameyear}

\end{document}